\documentclass{amsart}
\usepackage[letterpaper, margin=1.4in]{geometry}
\usepackage{amssymb, latexsym, MnSymbol}
\usepackage{mathtools}
\usepackage{color, xcolor}
\usepackage{amscd, graphicx, psfrag, tikz}
\usepackage[mathscr]{euscript}
\usepackage[all]{xy}

\usepackage{enumitem}
\usepackage{stmaryrd}
\usepackage{url}
\usepackage{hyperref}
\hypersetup{
    colorlinks=true,
    linkcolor=black,
    citecolor=black,
    urlcolor=black
}
\usepackage{verbatim}

\newcommand{\bC}{ {\mathbb{C}} }

\newcommand{\bP}{\mathbb{P}}

\newcommand{\bR}{\mathbb{R}}

\newcommand{\bT}{\mathbb{T}}

\newcommand{\bZ}{\mathbb{Z}}

\newcommand{\cL}{\mathcal{L}}
\newcommand{\cM}{\mathcal{M}}
\newcommand{\cO}{\mathcal{O}}

\newcommand{\cT}{\mathcal{T}}
\newcommand{\cX}{\mathcal{X}}

\newcommand{\Hom}{\mathrm{Hom}}

\newcommand{\Spec}{\mathrm{Spec}}

\newcommand{\ev}{\mathrm{ev}}

\newcommand{\vir}{ {\mathrm{vir}} }

\newcommand{\pt}{\mathrm{pt}}

\newcommand{\NE}{{\mathrm{NE}}}

\newcommand{\Id}{\mathrm{Id}}
\newcommand{\Tot}{\mathrm{Tot}}
\newcommand{\red}{\mathrm{red}}

\newcommand{\one}{\mathbf{1}}

\newcommand{\bw}{\mathbf{w}}

\newcommand{\su}{\mathsf{u}}
\newcommand{\sv}{\mathsf{v}}

\newcommand{\sX}{\mathsf{X}}

\newcommand{\hH}{\hat{H}}
\newcommand{\hM}{\hat{M}}

\newcommand{\tGa}{\widetilde{\Gamma}}

\newcommand{\tphi}{\widetilde{\phi}}
\newcommand{\tSi}{\widetilde{\Sigma}}
\newcommand{\trho}{\widetilde{\rho}}

\newcommand{\tA}{\widetilde{A}}
\newcommand{\tB}{\widetilde{B}}
\newcommand{\tC}{\widetilde{C}}
\newcommand{\tD}{\widetilde{D}}

\newcommand{\tM}{\widetilde{M}}
\newcommand{\tN}{\widetilde{N}}

\newcommand{\tQ}{\widetilde{Q}}

\newcommand{\tT}{\widetilde{T}}

\newcommand{\tX}{\widetilde{X}}

\newcommand{\tb}{\widetilde{b}}
\newcommand{\tc}{\widetilde{c}}

\newcommand{\tg}{\widetilde{g}}

\newcommand{\tp}{\widetilde{p}}

\newcommand{\tgamma}{\widetilde{\gamma}}
\newcommand{\tsi}{\widetilde{\sigma}}
\newcommand{\tbeta}{\widetilde{\beta}}

\newcommand{\ttau}{\widetilde{\tau}}

\newcommand{\tbw}{\widetilde{\bw}}

\newcommand{\vd}{\vec{d}}
\newcommand{\vf}{\vec{f}}

\newcommand{\vs}{\vec{s}}

\newcommand{\Mbar}{\overline{\cM}}

\newcommand{\Sbar}{\overline{S}}

\newcommand{\inner}[1]{\langle  #1 \rangle}
\newcommand{\formal}[1]{\llbracket  #1 \rrbracket}

\newtheorem{dummy}{dummy}[section]
\newtheorem{lemma}[dummy]{Lemma}
\newtheorem{theorem}[dummy]{Theorem}

\newtheorem{proposition}[dummy]{Proposition}
\newtheorem{remark}[dummy]{Remark}
\newtheorem{definition}[dummy]{Definition}

\newtheorem{assumption}[dummy]{Assumption}

\begin{document}
\title{Open WDVV equations and Frobenius structures for toric Calabi-Yau 3-folds}

\author{Song Yu}
\address{Song Yu, Yau Mathematical Sciences Center, Tsinghua University, Haidian District, Beijing 100084, China}
\email{song-yu@tsinghua.edu.cn}

\author{Zhengyu Zong}
\address{Zhengyu Zong, Department of Mathematical Sciences,
Tsinghua University, Haidian District, Beijing 100084, China}
\email{zyzong@mail.tsinghua.edu.cn}

\begin{abstract}
Let $X$ be a toric Calabi-Yau 3-fold and let $L\subset X$ be an Aganagic-Vafa outer brane. We prove two versions of open WDVV equations for the open Gromov-Witten theory of $(X,L)$. The first version of the open WDVV equation leads to the construction of a semi-simple (formal) Frobenius manifold and the second version leads to the construction of a flat (formal) $F$-manifold.
\end{abstract}
\maketitle

\setcounter{tocdepth}{1}
\tableofcontents

\section{Introduction}\label{sect:Intro}

\subsection{Historical background and motivation}

\subsubsection{WDVV equations and Frobenius manifolds}
The \emph{Witten-Dijkgraaf-Verlinde-Verlinde} \emph{(WDVV) equation} is a system of non-linear partial differential equations for one function, depending on a finite number of variables. One of the most important applications of the WDVV equation is the study of the quantum cohomology of a smooth projective variety $\cX$ over $\bC$. Let $\{T_i\}_{i=1}^m$ be a basis of $H^*(\cX)$\footnote{In this paper, $H^*(-)$ takes $\bC$-coefficients unless otherwise specified.} and $t^1,\dots,t^m$ be the corresponding coordinates. Let
$$
g_{ij}=( T_i,T_j  )_{\cX}=\int_{\cX}T_i\cup T_j
$$
and $(g^{ij})=(g_{ij})^{-1}$. Let $F_0^{\cX}$ be the generating function of genus-zero Gromov-Witten invariants of $\cX$ which depends on the variables $t^1,\dots,t^m$. The following theorem is the WDVV equation in Gromov-Witten theory, first proved in \cite{MS94, RT94}.

\begin{theorem}[\cite{MS94, RT94}]\label{thm:WDVV general}
For any $i,j,k,l\in\{1,\dots,m\}$, the following WDVV equation holds
$$
\frac{\partial^3 F_0^{\cX}}{\partial t^i\partial t^j\partial t^\nu}
\cdot g^{\nu\mu}\cdot \frac{\partial^3 F_0^{\cX}}{\partial t^\mu\partial t^k\partial t^l}
=\frac{\partial^3 F_0^{\cX}}{\partial t^j\partial t^k\partial t^\nu}
\cdot g^{\nu\mu}\cdot \frac{\partial^3 F_0^{\cX}}{\partial t^\mu\partial t^i\partial t^l}.
$$
\end{theorem}



The importance of the WDVV equation is that it implies the associativity of the \emph{quantum product} $\star_t$ defined by
$$
(T_i\star_t T_j,T_k)_{\cX}=\frac{\partial^3 F_0^{\cX}}{\partial t^i\partial t^j\partial t^k}
$$
for $i,j\in\{1,\dots,m\}$. The associativity of the quantum product has many important applications. A typical example is the simple, recursive formula given by Kontsevich and Manin \cite{KM94} that calculates the Gromov-Witten invariants of $\bP^2$. The geometric insight behind the formula is a splitting principle which is captured by the associativity of the quantum product. The WDVV equation and Kontsevich-Manin axioms were then used by G\"{o}ttsche and Pandharipande \cite{GP98} to give a set of formulae that recursively compute the Gromov-Witten invariants of $\bP^2_r$, the blowup of $\bP^2$ at $r$ points.


Moreover, the quantum product determines the structure of a \emph{Frobenius manifold}. 

\begin{definition}\label{def:Frobenius}
A \emph{complex Frobenius manifold} consists of the data $(M,g,A,\one)$ where
\begin{enumerate}
\item $M$ is a complex manifold of dimension $m$;
\item $g$ is a flat holomorphic metric on the tangent bundle $\cT_M$;
\item $A$ is a holomorphic tensor
$$
A:\cT_M\otimes \cT_M\otimes \cT_M\to \cO_M,
$$
where $\cO_M$ is the sheaf of holomorphic functions on $M$.
\item $\one$ is a holomorphic vector field on $M$.
\end{enumerate}
The above data are required to satisfy the following conditions.
\begin{enumerate}
\item (Potentiality) $M$ is covered by open sets $U$ each equipped with a commuting basis of $g-$flat holomorphic vector fields,
$$
X_1,\dots,X_m\in \cT_M(U)
$$
and a holomorphic potential function $F\in \cO_U(U)$ such that
$$
A(X_i,X_j,X_k)=X_iX_jX_k(F).
$$

\item (Associativity) Define a commutative product $\star$ on $\cT_M$ by
$$
g(X\star Y,Z)=A(X,Y,Z)
$$
where $X,Y,Z$ are holomorphic vector fields. Then we require that $\star$ is associative.
\item (Unit) $\one$ is $g-$flat and is a unit for the product $\star$.
\end{enumerate}

\end{definition}

The structure of Frobenius manifolds appears in different areas of mathematics including the singularity theory and curve counting theories in algebraic geometry (Gromov-Witten theory, Fan-Jarvis-Ruan-Witten theory). A systematic study of Frobenius manifolds was first done by Dubrovin \cite{Dub96, Dub99}. Again, the associativity of the product $\star$ is equivalent to the fact that the potential function $F$ in Definition \ref{def:Frobenius} satisfies the WDVV equation in Theorem \ref{thm:WDVV general} by replacing $\frac{\partial}{\partial t^i}$ by $X_i$. Let $\nabla$ be the Levi-Civita connection corresponding to the metric $g$. For $z\in\bP^1$, define the \emph{Dubrovin connection} $\nabla^z$ as
$$
\nabla_{X}^z(Y)=\nabla_{X}(Y)-\frac{1}{z}X\star Y.
$$
Then it is easy to see that the associativity of $\star$ is equivalent to the flatness of $\nabla^z$ and that the commutativity of $\star$ is equivalent to the fact that $\nabla^z$ is symmetric.

In the case of quantum cohomology, suppose that the genus-zero Gromov-Witten potential $F_0^{\cX}$ is convergent in a neighborhood $U$ of the origin. One may take $M$ to be $U$ and the potential function $F$ to be $F_0^{\cX}$ in Definition \ref{def:Frobenius}. Moreover, let the metric $g$ be given by the Poincar\'{e} pairing on $H^*(\cX)$ and $\one$ be the identity in $H^*(\cX)$. Then one obtains a Frobenius manifold. In general, the genus-zero Gromov-Witten potential $F_0^{\cX}$ is not convergent. Then one can replace the above formalism by considering \emph{formal} Frobenius manifolds (see \cite{Manin99,LP}). Specifically, one can replace the complex manifold $M$ by the formal scheme $\hat{H} :=\mathrm{Spec}(\Lambda_{\cX}\formal{t^1,\dots,t^m})$ over the base ring $\Lambda_{\cX}$ which is the Novikov ring of $\cX$. Then one may view $F_0^{\cX}$ as a regular function on $\hat{H}$ and obtain a formal Frobenius manifold. See Section \ref{sect:closedFrob} for additional details, including definitions of formal Frobenius manifolds over general base rings.

\subsubsection{Open WDVV equations and $F$-manifolds}\label{sec:F-mfd}
The \emph{open WDVV equation} is a system of non-linear partial differential equations that extends the WDVV equation by introducing an additional variable $t^o$ for the open sector and an additional potential function $F^o(t^1,\dots,t^m,t^o)$ called the \emph{open potential function}. One of the most important motivations to introduce the open WDVV equation is to study open Gromov-Witten theory. In \cite{Solomon07,HS,ST19}, open Gromov-Witten invariants of $(\cX,\cL)$ are studied for certain symplectic manifolds $\cX$ and Lagrangian submanifolds $\cL\subset\cX$. In these cases, one can introduce the \emph{disk potential} $F^{\cX,\cL}_{0,1}$ which is the generating function of disk Gromov-Witten invariants of $(\cX,\cL)$. As before, let $\{T_i\}_{i=1}^m$ be a basis of $H^*(\cX)$ and $t^1,\dots,t^m$ be the corresponding coordinates. We still consider the Poincar\'{e} pairing $g_{ij}=( T_i,T_j)_{\cX}=\int_{\cX}T_i\cup T_j$ and let $(g^{ij})=(g_{ij})^{-1}$. Let $F_0^{\cX}$ be the generating function of genus-zero Gromov-Witten invariants of $\cX$, which depends on the variables $t^1,\dots,t^m$ but is independent of the additional variable $t^o$. On the other hand, the disk potential $F^{\cX,\cL}_{0,1}$ depends on $t^1,\dots,t^m$ as well as $t^o$. The variable $t^o$ encodes the point-like insertions from the boundary marked points of the domain disk (see \cite{HS, ST19} for more details). The following open WDVV equation is proved in \cite{HS,ST19}.

\begin{theorem}[\cite{HS,ST19}]\label{thm:openWDVV point}
For any $i,j,k\in\{1,\dots,m\}$, the following open WDVV equation holds:
\begin{eqnarray*}
\frac{\partial^3 F_0^{\cX}}{\partial t^i\partial t^j\partial t^\mu}g^{\mu\nu}\frac{\partial^2 F^{\cX,\cL}_{0,1}}{\partial t^\nu\partial t^k} +\frac{\partial^2 F^{\cX,\cL}_{0,1}}{\partial t^i\partial t^j}\frac{\partial^2 F^{\cX,\cL}_{0,1}}{\partial t^o\partial t^k} &=&
\frac{\partial^3 F_0^{\cX}}{\partial t^k\partial t^j\partial t^\mu}g^{\mu\nu}\frac{\partial^2 F^{\cX,\cL}_{0,1}}{\partial t^\nu\partial t^i}+\frac{\partial^2 F^{\cX,\cL}_{0,1}}{\partial t^k\partial t^j}\frac{\partial^2 F^{\cX,\cL}_{0,1}}{\partial t^o\partial t^i},\\
\frac{\partial^3 F_0^{\cX}}{\partial t^i\partial t^j\partial t^\mu}g^{\mu\nu}\frac{\partial^2 F^{\cX,\cL}_{0,1}}{\partial t^\nu\partial t^o}+\frac{\partial^2 F^{\cX,\cL}_{0,1}}{\partial t^i\partial t^j}\frac{\partial^2 F^{\cX,\cL}_{0,1}}{(\partial t^o)^2}&=&
 \frac{\partial^2 F^{\cX,\cL}_{0,1}}{\partial t^o\partial t^j}\frac{\partial^2 F^{\cX,\cL}_{0,1}}{\partial t^o\partial t^i}.
\end{eqnarray*}

\end{theorem}

The open WDVV equation has also been studied in \cite{PST14, Alcolado17,BCT18,BCT19,BB19,CZ19,ABLR20}.

The natural structure that captures the open WDVV equation is that of a \emph{flat $F$-manifold} (see, for example, \cite{HM99,Manin99,Getzler04,Manin05,ABLR20}), a generalization of a Frobenius manifold.
\begin{definition}\label{def:F-mfd}
A \emph{flat complex $F$-manifold} consists of the data $(M,\nabla,\star,\one)$ where
\begin{enumerate}
\item $M$ is a complex manifold of dimension $m+1$,
\item $\nabla$ is a holomorphic connection on the tangent bundle $\cT_M$,
\item $(\cT_M \big|_p,\star)$ defines an algebra structure on each tangent space, analytically depending on the point $p\in M$,
\item $\one$ is a $\nabla$-flat vector field which is a unit for $\star$.
\end{enumerate}
The above data satisfy the condition that the connection $\nabla^z:=\nabla-\frac{1}{z}\star$ is flat and symmetric for any $z\in\bP^1$.
\end{definition}

The structure of $F$-manifolds appears in different areas of mathematics including the open Gromov-Witten theory, Painlev\'{e} transcendents, and reflection groups. Again, the associativity of $\star$ is equivalent to the flatness of $\nabla^z$ and that the commutativity of $\star$ is equivalent to the fact that $\nabla^z$ is symmetric. Moreover, if one chooses flat coordinates $t^1,\dots,t^{m+1}$ for the connection $\nabla$, then it is easy to see that locally there exist holomorphic functions $F^i(t^1,\dots,t^{m+1})$, $i = 1, \dots, m+1$, such that the second derivatives
$$
c^{i}_{jk}:=\frac{\partial^2 F^i}{\partial t^j\partial t^k}
$$
are the structure constants of the algebra $(\cT_M \big|_p,\star)$:
$$
\frac{\partial}{\partial t^j}\star\frac{\partial}{\partial t^k}=c^{i}_{jk}\frac{\partial}{\partial t^i}.
$$
Then the associativity of $\star$ is equivalent to the equation
\begin{equation}\label{eqn:F-mfd}
\frac{\partial^2 F^i}{\partial t^j\partial t^\mu}\frac{\partial^2 F^\mu}{\partial t^k\partial t^l}=\frac{\partial^2 F^i}{\partial t^k\partial t^\mu}\frac{\partial^2 F^\mu}{\partial t^j\partial t^l}
\end{equation}
for $i,j,k,l\in\{1,\dots,m+1\}$. The $(m+1)-$tuple $\overline{F}=(F^1,\dots,F^{m+1})$ is called the \emph{vector potential} for the $F$-manifold $M$. In the special case when $M$ is a Frobenius manifold with potential $F$, and assuming that $g = \Id$ for simplicity, the vector potential is given by $\overline{F}=(\frac{\partial F}{\partial t^1},\dots,\frac{\partial F}{\partial t^{m+1}})$.

In the case of the open Gromov-Witten theory of $(\cX, \cL)$ with point-like boundary insertions, one can construct a flat $F$-manifold as follows. First we can choose $\{T_i\}_{i=1}^m$ such that $g=\Id$. Let $\nabla$ be the connection under which $\frac{\partial}{\partial t^1},\dots,\frac{\partial}{\partial t^m},\frac{\partial}{\partial t^o}$ are flat. Finally, define the vector potential by
$$
    \overline{F} := (\frac{\partial F_0^{\cX}}{\partial t^1},\dots,\frac{\partial F_0^{\cX}}{\partial t^m},F^{\cX,\cL}_{0,1}).
$$
In the case when $\overline{F}$ is convergent, we obtain a complex $F$-manifold of dimension $m+1$. Equation \eqref{eqn:F-mfd} is obtained by the open and closed WDVV equations (Theorems \ref{thm:WDVV general} and \ref{thm:openWDVV point}). In general, the vector potential $\overline{F}$ is not convergent, and one can construct a \emph{formal} $F$-manifold as in the case of the quantum cohomology.

\subsubsection{Open WDVV equations for toric Calabi-Yau 3-folds}
In this paper, we study the open WDVV equation for $(X,L)$ where $X$ is a toric Calabi-Yau 3-fold and $L\subset X$ is an outer Aganagic-Vafa brane. We first obtain a collection of non-linear partial differential equations (Proposition \ref{prop:openWDVV}) which involve both the generating function $F_0^{X,T'}$ of genus-zero equivariant Gromov-Witten invariants of $X$ and the generating function $F_{0,1}^{X,(L,f)}$ of equivariant disk Gromov-Witten invariants of $(X,L)$ . We will package these equations in two different ways to obtain two versions of the open WDVV equation. The first version leads to the construction of a semi-simple formal Frobenius manifold, and the second version leads to a flat formal $F$-manifold, both exhibiting the recursive structures of the open and closed Gromov-Witten theory of $(X,L)$.

The key technique we use to derive the open WDVV equation is the \emph{open/closed correspondence} \cite{LY21,LY22} which relates the open Gromov-Witten theory of $(X,L)$ and the closed Gromov-Witten theory of a corresponding toric Calabi-Yau 4-fold $\tX$. Based on the original conjectures of Mayr \cite{Mayr01} in physics, the mathematical development of the correspondence emerges from studies of correspondences among different types (open, relative/log, local) of Gromov-Witten invariants in the literature \cite{LLLZ09,FL13,vGGR19,BBvG24,GRZ22}. Under this correspondence, we may recover both $F_{0,1}^{X,(L,f)}$ and $F_0^{X,T'}$ from the generating function $F_0^{\tX,\tT'}$ of genus-zero equivariant Gromov-Witten invariants of $\tX$. The open WDVV equation for $(X,L)$ is then a consequence of the usual WDVV equation for $\tX$. Recently, the open/closed correspondence has also been applied to study the integrality properties of Gromov-Witten invariants of $(X,L)$ and $\tX$ (in terms of BPS or Gopakumar-Vafa invariants) \cite{Yu23}. The correspondence has also been studied on the B-model side of mirror symmetry \cite{LY22} and extended to quintic 3-folds \cite{AL23}.

We now discuss our main results and techniques in more detail.

\subsection{Statement of the main results}\label{sec:main result}

Let $X$ be a smooth toric Calabi-Yau 3-fold and $T\cong (\bC^*)^3$ be the algebraic 3-torus embedded in $X$ as a dense open subset. Let $T' \cong (\bC^*)^2$ be the \emph{Calabi-Yau $2$-subtorus} of $T$ which acts trivially on the canonical bundle of $X$. Let $L \subset X$ be an \emph{Aganagic-Vafa outer brane} in $X$ which is a Lagrangian submanifold diffeomorphic to $S^1 \times \bC$. It intersects a unique $T$-invariant line $l\cong \bC$ in $X$. Moreover, $L$ is invariant under the action of the maximal compact subtorus $T_{\bR}' \cong U(1)^2$ of $T'$. We further take an integer $f$ called the \emph{framing} on the Aganagic-Vafa brane $L$, and construct a 1-dimensional subtorus $T_f\subset T'$.

Under the open/closed correspondence, the closed geometry corresponding to the open geometry $(X,L,f)$ is a smooth toric Calabi-Yau 4-fold $\tX$ that takes the form
$$
    \tX = \Tot(\cO_{X \sqcup D}(-D)),
$$
where $X \sqcup D$ is a toric partial compactification of $X$ given by adding an additional toric divisor $D$. In $X \sqcup D$, the $T$-invariant line $l \cong \bC$ that $L$ intersects is compactified by an additional $T$-fixed point into a $\bP^1$ whose normal bundle is isomorphic to $\cO_{\bP^1}(f) \oplus \cO_{\bP^1}(-f-1)$. There is an inclusion
$$
    X \to X \sqcup D \to \tX.
$$
Let $\tT \cong (\bC^*)^4$ be the algebraic 4-torus of $\tX$ and $\tT' \cong (\bC^*)^3$ be the Calabi-Yau 3-subtorus of $\tT$, which contains $T'$ as a subtorus. We take the following notations for the equivariant parameters of the tori:
\begin{align*}
    & R_{\tT'} := H^*_{T'}(\pt) = \bC[\su_1, \su_2, \su_4], && S_{\tT'}:= \bC(\su_1, \su_2, \su_4),\\
    & R_{T'} := H^*_{T'}(\pt) = \bC[\su_1, \su_2], && S_{T'}:= \bC(\su_1, \su_2),\\
    & R_{T_f} := H^*_{T_f}(\pt) = \bC[\su_1], && S_{T_f}:= \bC(\su_1).
\end{align*}

Let $p_1, \dots, p_m$ be a fixed ordering of the $T'$-fixed points of $X$ and $\tp_1, \dots, \tp_m$ denote the corresponding $\tT'$-fixed points of $\tX$. We denote the additional $\tT'$-fixed point of $\tX$ by $\tp_{m+1}$. We consider the basis $\{\phi_1,\dots,\phi_m\}$ of $H^*_{T'}(X) \otimes_{R_{T'}} S_{T'}$ defined by the fixed points as
$$
    \phi_i := \frac{[p_i]}{e_{T'}(T_{p_i}X)},
$$
which forms a canonical basis of the semi-simple Frobenius algebra
$$
    (H^*_{T'}(X) \otimes_{R_{T'}} S_{T'},\cup,(-,-)_{X,T'})
$$
where $\cup$ is the cup product and $(-,-)_{X, T'}$ is the $T'$-equivariant Poincar\'e pairing on $X$. Similarly, we define the basis $\{\tphi_1,\dots,\tphi_m,\tphi_{m+1} \}$ of $H^*_{\tT'}(\tX) \otimes_{R_{\tT'}} S_{\tT'}$ as
$$
    \tphi_i := \frac{[\tp_i]}{e_{\tT'}(T_{\tp_i}\tX)},
$$
which forms a canonical basis of the semi-simple Frobenius algebra
$$
    (H^*_{\tT'}(\tX) \otimes_{R_{\tT'}} S_{\tT'},\cup,(-,-)_{\tX,\tT'}).
$$
Let $t^1,\dots,t^m,t^{m+1}$ be the coordinates corresponding to the basis $\{\tphi_1,\dots,\tphi_m,\tphi_{m+1} \}$. Under the correspondence between $\{\tphi_1,\dots,\tphi_m\}$ and $\{\phi_1,\dots,\phi_m\}$, we also view $t^1,\dots,t^m$ as coordinates corresponding to the basis $\{\phi_1,\dots,\phi_m\}$.

We will use the above bases to define the following generating functions of Gromov-Witten invariants over suitable Novikov rings:
\begin{itemize}
    \item $F_0^{X,T'}(t^1, \dots, t^m)$ -- the generating function of genus-zero $T'$-equivariant closed Gromov-Witten invariants of $X$;

    \item $F_{0,1}^{X,(L,f)}(t^1, \dots, t^m, t^o)$ -- the generating function of genus-zero $T_f$-equivariant disk invariants of $(X,L)$ with framing $f$, depending on an additional formal variable $t^o$ for the open sector;

    \item $F_0^{\tX,\tT'}(t^1, \dots, t^m, t^{m+1})$ -- the generating function of genus-zero $\tT'$-equivariant closed Gromov-Witten invariants of $\tX$.
\end{itemize}
See Section \ref{sect:GW} for detailed definitions. In particular, we will see that the dependence of $F_{0,1}^{X,(L,f)}$ on the additional open variable $t^o$ is captured by terms of form $(e^{t^o}\sX_0)^d$, $d \in \bZ_{>0}$, where $\sX_0$ is a Novikov variable for the relative curve class. The term $\sX=e^{t^o}\sX_0$ is viewed as encoding the winding number $d$ of the disk invariants.

The open/closed correspondence (see Theorem \ref{thm:OpenClosed}) retrieves both $F_0^{X,T'}$ and $F_{0,1}^{X,(L,f)}$ from $F_0^{\tX,\tT'}$ under a suitable change of coordinates and Novikov variables. The WDVV equation for $F_0^{\tX,\tT'}$ (Theorem \ref{thm:WDVV general}) then gives rise to a collection of non-linear partial differential equations involving $F_0^{X,T'}$ and $F_{0,1}^{X,(L,f)}$ (see Proposition \ref{prop:openWDVV}). This collection recovers the WDVV equation for $F_0^{X,T'}$, and contains equations analogous to the open WDVV equation obtained by \cite{HS,ST19} (Theorem \ref{thm:openWDVV point}). As our main results, we use this collection of equations to construct a semi-simple formal Frobenius manifold and a flat formal $F$-manifold to package the structures of the open and closed Gromov-Witten theory of $(X,L)$.

\subsubsection{A formal Frobenius manifold}
The first aspect of our constructions is a formal Frobenius manifold (Section \ref{sect:FrobStr}). Consider the formal scheme
$$
    \hH_1 :=\mathrm{Spec}(\Lambda_{X,L}^{T_f}[\epsilon]\formal{t^1,\dots,t^m ,t^o})
$$
over the base ring
$$
    \Lambda_{X,L}^{T_f}[\epsilon] := \Lambda_{X,L}^{T_f} \otimes \bC[\epsilon]/\inner{\epsilon^2},
$$
where $\Lambda_{X,L}^{T_f}$ is the $T_f$-equivariant Novikov ring of $(X,L)$ and $\epsilon$ is a nilpotent variable with $\epsilon^2 = 0$. We will define a pairing $\left(\frac{\partial}{\partial t^i},\frac{\partial}{\partial t^j}\right) = h_{ij}$, $i,j = 1, \dots, m, o$, on the tangent bundle $\cT_{\hH_1}$ of $\hat{H}_1$ which is spanned by vector fields $\frac{\partial}{\partial t^1},\dots,\frac{\partial}{\partial t^m}, \frac{\partial}{\partial t^o}$. Let $(h^{ij})=(h_{ij})^{-1}$. Moreover, we define the potential function $F$ by
$$
    F:= -\frac{\su_1}{6}(t^o)^3 + F_0^{X,T'}\big|_{\su_2-f\su_1=0} + \epsilon \int F_{0,1}^{X,(L,f)},
$$
where the weight restriction $\su_2 - f\su_1 = 0$ corresponds to the inclusion $T_f \subset T'$ and the symbol $\int$ represents taking the antiderivative with respect to $t^o$. We show that $F$ satisfies the following WDVV equation.

\begin{proposition}[See Proposition \ref{prop:openWDVV1}]
For any $i,j,k,l\in\{1,\dots,m, o\}$, the following WDVV equation holds:
$$
    \frac{\partial^3 F}{\partial t^i\partial t^j\partial t^\nu}
    \cdot h^{\nu\mu}\cdot \frac{\partial^3 F}{\partial t^\mu\partial t^k\partial t^l}
    =\frac{\partial^3 F}{\partial t^j\partial t^k\partial t^\nu}
    \cdot h^{\nu\mu}\cdot \frac{\partial^3 F}{\partial t^\mu\partial t^i\partial t^l}.
$$
\end{proposition}

In particular, the potential $F$ defines a product $\star_t$ on $\cT_{\hH_1}$ that is compatible with the metric $h$ and associative. We have the following main structural result.

\begin{theorem}[See Theorems \ref{thm:Frob}, \ref{thm:decomp}]
The tuple $(\hat{H}_1,F,(-,-))$ is a semi-simple formal Frobenius manifold over $\Lambda_{X,L}^{T_f}[\epsilon]$.
\end{theorem}

\begin{remark}\label{rmk:Frob} \rm{
One way to interpret the variable $\epsilon$ is the following. Consider $\hH_1$ as a formal \emph{supermanifold} over $\Lambda_{X,L}^{T_f}$ with local coordinates $t^1,\dots,t^m,t^o,\epsilon$ where $t^1,\dots,t^m,t^o$ are even coordinates and $\epsilon$ is an odd coordinate (and hence $\epsilon^2=0$). Then the pairing $h$ and the product structure $\star_t$ may be viewed as defined on the subbundle of the tangent bundle spanned by the even vector fields $\frac{\partial}{\partial t^1},\dots,\frac{\partial}{\partial t^m}, \frac{\partial}{\partial t^o}$. The product $\star_t$ itself does not involve the odd vector field $\frac{\partial}{\partial \epsilon}$. Rather, as remarked in e.g. \cite[Section 4.1]{MM97}, $\epsilon$ is regarded as an odd structural constant pulled back from the base $\Spec (\Lambda_{X,L}^{T_f}[\epsilon])$ viewed also as a supermanifold over $\Lambda_{X,L}^{T_f}$.
}\end{remark}

\subsubsection{A flat formal $F$-manifold} The second aspect of our constructions is a flat formal $F$-manifold (Section \ref{sect:F-Str}). Consider the formal scheme
$$
    \hH_2 :=\mathrm{Spec}(\Lambda_{X,L}^{T_f}\formal{t^1,\dots,t^m ,t^o})
$$
over the base ring $\Lambda_{X,L}^{T_f}$, where as compared to $\hH_1$ above, the variable $\epsilon$ is dropped. Let $\nabla$ be the flat connection on the tangent bundle $\cT_{\hH_2}$ of $\hH_2$ under which $\frac{\partial}{\partial t^1},\dots,\frac{\partial}{\partial t^m},\frac{\partial}{\partial t^o}$ are flat. We define the vector potential $\overline{F} = (F^1, \dots, F^m, F^o)$ by
\begin{align*}
    & F^i := h^{ii} \frac{\partial}{\partial t^i} \left(F_0^{X,T'}\big|_{\su_2-f\su_1=0}+\int F_{0,1}^{X,(L,f)} \big|_{t^o = 0}\right), && i = 1, \dots, m,\\
    &F^o := F_{0,1}^{X,(L,f)} \big|_{t^o = 0}.
\end{align*}
Here $t^o$ is still viewed as the variable for the `open state space', while we should notice that each component of $\overline{F}$ is independent of $t^o$. We show that $\overline{F}$ satisfies the following open WDVV equation.

\begin{proposition}[See Proposition \ref{prop:openWDVV2}]
For any $i,j,k,l\in\{1,\dots,m, o\}$, the following open WDVV equation holds:
$$
    \frac{\partial^2 F^j}{\partial t^i\partial t^\mu}
    \cdot  \frac{\partial^2 F^\mu}{\partial t^k\partial t^l}
    =\frac{\partial^2 F^j}{\partial t^k\partial t^\mu}
    \cdot  \frac{\partial^2 F^\mu}{\partial t^i\partial t^l}.
$$
\end{proposition}

In particular, the vector potential $\overline{F}$ defines a product structure $\star_t$ on $\cT_{\hH_2}$ that is associative. Analyzing the structural constants and using that $\overline{F}$ is independent of $t^o$, we show that $\star_t$ does not admit an identity field and $\frac{\partial}{\partial t^o}$ is nilpotent. We arrive at the following main structural result.

\begin{theorem}[See Theorem \ref{thm:F}]
The tuple $(\hH_2, \nabla, \star_t)$ is a flat formal $F$-manifold without unit over $\Lambda_{X,L}^{T_f}$ in which the $t^o$-direction is nilpotent.
\end{theorem}

$F$-cohomological field theories without unit have been studied in \cite{ABLR20,BG23}.

\begin{remark}\label{rmk:F} \rm{
The situation here is in a sense opposite to that in Remark \ref{rmk:Frob}: in the Frobenius manifold $\hH_1$, the variable $\epsilon$ appears in the potential $F$ while $\frac{\partial}{\partial \epsilon}$ is not involved in the product $\star_t$; in the $F$-manifold $\hH_2$, the variable $t^o$ does not appear in $\overline{F}$ while $\frac{\partial}{\partial t^o}$ is involved in $\star_t$. From a geometric point of view, we may view the open variable $t^o$ as parameterizing a divisor-like insertion arising from the open sector. In $\hH_1$, it contributes to the factor $e^{t^o}\sX_0$ appearing in the $F_{0,1}^{X,(L,f)}$-part of $F$ via the ``open divisor equation''. On the other hand, in $\hH_2$, the vector potential $\overline{F}$ defined by the restriction $t^o = 0$ has no boundary insertions and does not depend on $t^o$.
}\end{remark}

Despite the above differences, we will see that both structures $\hH_1$ and $\hH_2$ can be viewed as extensions of the formal Frobenius manifold determined by $F_0^{X,T'}$; see Remarks \ref{rmk:FrobDecomp}, \ref{rmk:FExtension}.

\subsection{Future works}
\subsubsection{Recursion for open Gromov-Witten invariants of toric Calabi-Yau 3-folds}
In \cite{KM94}, Kontsevich and Manin proved that closed Gromov-Witten invariants can be recursively computed
from an initial set of known values. In particular, when $X$ is Fano, this initial set of values is finite. A typical example is the \emph{recursive formula} that calculates the Gromov-Witten invariants of $\bP^2$. This theorem is proved via the WDVV equations for closed Gromov-Witten invariants.

In the study of open Gromov-Witten invariants, similar recursive formulas can be obtained via open WDVV equations. In many cases \cite{CZ19, GZ17, HKSSS23, HS, HT24, ST19}, open Gromov-Witten invariants have been shown to be computable from a finite initial
set of values. Later in \cite{BT24}, a more general recursive formula is obtained based on a formal object called the Frobenius superpotential.

In our case, the open WDVV equations can be used to prove a recursive formula for open Gromov-Witten invariants of toric Calabi-Yau 3-folds. Since we study equivariant Gromov-Witten theory and our target spaces are Calabi-Yau, this recursive formula is more subtle and contains richer structures.

\subsubsection{Open-closed map and variation of Hodge structures}
In his ICM address, Kontsevich conjectured the homological mirror symmetry and moreover conjectured that this homological mirror symmetry implies enumerative mirror symmetry. Ganatra-Perutz-Sheridan \cite{GPS15} show that for certain Calabi-Yaus, whose variations of Hodge structures are of Hodge-Tate type, the genus-zero Gromov-Witten invariants are indeed extractable from the Fukaya category. The strategy is to show that the open-closed map \cite{FOOO09,Ganatra23} respects the variation of Hodge structures.

In \cite{H22}, the open-closed map is extended to a map from the relative cyclic homology to the \emph{relative quantum homology} whose definition is based on the open WDVV equations. In our case of toric Calabi-Yau 3-folds, the open WDVV equations studied in this paper may be used to construct relative quantum cohomology, which would further enable a study of the relative open-closed map and its compatibility with variations of Hodge structures.

\subsection{Outline of the paper}
In Section \ref{sect:Geometry}, we review the open geometry of $(X,L)$ and the corresponding closed geometry of $\tX$. We will also study the equivariant cohomology of $X$ and $\tX$. In Section \ref{sect:GW}, we give the basic definitions of open and closed Gromov-Witten invariants for $X$ and $\tX$, and then state the open/closed correspondence in Section \ref{sect:OpenClosed}. In Section \ref{sect:closedFrob}, we review the WDVV equation in closed Gromov-Witten theory and  use the specialization to $\tX$ to prove non-linear partial differential equations which involve $F_0^{X,T'}$ and $F_{0,1}^{X,(L,f)}$. Finally, in Section \ref{sect:openFrob}, we use these equations to establish the main results of the paper on the formal Frobenius and $F$-manifold structures.


\subsection*{Acknowledgments}
The authors wish to thank Alexandr Buryak, Bohan Fang, Sheel Ganatra, Chiu-Chu Melissa Liu, Jake Solomon, Junwu Tu, Ke Zhang, and Yang Zhou for useful discussions and constructive feedback. The authors also wish to thank the hospitality of the Simons Center for Geometry and Physics during the 2023 Simons Math Summer Workshop where part of this work was completed. The work of the second named author is partially supported by NSFC grant No. 11701315.


\section{Geometric setup}\label{sect:Geometry}
In this section, we review the geometry of toric Calabi-Yau 3-folds and Aganagic-Vafa branes. We then review the geometry of the corresponding toric Calabi-Yau 4-folds. We refer to \cite{FL13,FLT22,LY21,LY22} for additional details. We work over $\bC$.

\subsection{Notations for toric geometry}
In this paper, we use the following notations for an $r$-dimensional smooth toric variety $Z$ defined by a fan $\Xi$ in $\bR^r$. The algebraic torus of $Z$ is isomorphic to $(\bC^*)^r$.

\begin{itemize}
    \item For $d = 0, \dots, r$, let $\Xi(d)$ denote the set of $d$-dimensional cones in $\Xi$. For a cone $\sigma \in \Xi(d)$, let $V(\sigma) \subseteq Z$ denote the $(\bC^*)^r$-orbit closure corresponding to $\sigma$, which is a codimension-$d$ closed subvariety of $Z$.
    
    \item For a maximal cone $\sigma \in \Xi(r)$, let $p_\sigma := V(\sigma)$ denote the corresponding $(\bC^*)^r$-fixed point.
    
    \item For a cone $\tau \in \Xi(r-1)$, let $l_\tau := V(\tau)$ denote the corresponding $(\bC^*)^r$-invariant line, which is isomorphic to either $\bC$ or $\bP^1$. We set $\Xi(r-1)_c := \{\tau \in \Xi(r-1): l_\tau \cong \bP^1\}$.
    
    \item Let $F(\Sigma) := \{(\tau, \sigma) \in \Xi(r-1) \times \Xi(r) : \tau \subset \sigma\}$ denote the set of flags in $\Xi$.
\end{itemize}

\subsection{Open geometry}\label{sect:OpenGeometry}
Let $N \cong \bZ^3$ be a lattice and $M := \Hom(N, \bZ)$ be the dual lattice. Let $X$ be a smooth toric Calabi-Yau 3-fold specified by a finite fan $\Sigma$ in $N_{\bR} := N \otimes \bR \cong \bR^3$. We assume that $\Sigma(3)$ is non-empty and every cone in $\Sigma$ is a face of some 3-cone.

Let $R := |\Sigma(1)|$. Let $\Sigma(1) = \{\rho_1, \dots, \rho_R\}$ be a listing of the rays in $\Sigma$, and for each $i = 1, \dots, R$ let $b_i \in N$ be the primitive generator of $\rho_i$. The Calabi-Yau condition on $X$ is equivalent to the existence of $u_3 \in M$ such that $\inner{u_3, b_i} = 1$ for all $i$, where $\inner{-,-}$ is the natural pairing between $M$ and $N$. Let $N' := \ker(u_3: N \to \bZ) \cong \bZ^2$.

Let $P$ be the cross-section of the support $|\Sigma|$ of $\Sigma$ in the hyperplane
\begin{equation}\label{eqn:Hyperplane}
    \{v \in N_{\bR} : \inner{u_3, v} = 1\} \cong N' \otimes \bR \cong \bR^2,
\end{equation}
which is a 2-dimensional lattice polytope with a triangulation induced by $\Sigma$. We assume that $P$ is simple. As in the setup of \cite[Section 2.2]{LY21}, we do not assume that $P$ is convex or equivalently $X$ is semi-projective. There is a toric partial compactification $X \subseteq X'$ by a semi-projective smooth toric Calabi-Yau 3-fold $X'$ determined by a fan $\Sigma'$ which contains $\Sigma$ as a subfan and satisfies $\Sigma'(1) = \Sigma(1)$. The cross-section of $\Sigma'$ with the hyperplane \eqref{eqn:Hyperplane} is the convex hull $P'$ of $P$, and we have $P' \cap N = P \cap N = \{b_1, \dots, b_R\}$.

Let $T := N \otimes \bC^* \cong (\bC^*)^3$ be the algebraic torus of $X$, whose character lattice is $\Hom(T, \bC^*) \cong M$. We consider a 2-subtorus $T' := \ker(u_3: T \to \bC^*) = N' \otimes \bC^* \cong (\bC^*)^2$. The fixed points and invariant lines of $X$ under the $T'$-action are the same as those under the $T$-action.

Let $L \subset X$ be an \emph{Aganagic-Vafa brane} in $X$, which is a Lagrangian submanifold diffeomorphic to $S^1 \times \bC$. We refer to \cite[Section 2.4]{FL13}, \cite[Section 2.2]{LY21} for detailed definitions. The brane $L$ is invariant under the action of the maximal compact subtorus $T_{\bR}' \cong U(1)^2$ of $T'$. Moreover, it intersects a unique $T$-invariant line $l_{\tau_0}$ in $X$, where $\tau_0 \in \Sigma(2)$. Given a semi-projective toric partial compactification $X'$ of $X$ as above, $L$ can be viewed as an Aganagic-Vafa brane in $X'$, intersecting the $T$-invariant line in $X'$ corresponding to $\tau_0 \in \Sigma'(2)$. As in \cite[Assumption 2.3]{LY21}, we make the following assumption on $L$.

\begin{assumption}\label{assump:L} \rm{
We assume that $L$ is an \emph{outer} brane in the partial compactification $X'$, that is, $\tau_0 \in \Sigma'(2) \setminus \Sigma'(2)_c$.
}
\end{assumption}

Note that this assumption does not depend on the choice of $X'$. In particular, $\tau_0 \in \Sigma(2) \setminus \Sigma(2)_c$ and $L$ is also an outer brane in $X$. Let $\sigma_0 \in \Sigma(3)$ be the unique 3-cone containing $\tau_0$ as a face.

For any cone $\sigma$ in $\Sigma$, we set
$$
    I'_\sigma := \{i \in \{1, \dots, R\} : \rho_i \subseteq \sigma\}, \qquad I_\sigma := \{1, \dots, R\} \setminus I'_\sigma.
$$
We assume without loss of generality that
$$
    I'_{\tau_0} = \{2, 3\}, \qquad I'_{\sigma_0} = \{1, 2, 3\}
$$
with $b_1, b_2, b_3$ appearing in $P$ in a counterclockwise order. Such labeling determines a unique way to complete $u_3$ into a $\bZ$-basis $\{u_1, u_2, u_3\}$ of $M$ such that under the dual $\bZ$-basis $\{v_1, v_2, v_3\}$ of $N$, we have the coordinates
$$
    b_1 = (1, 0, 1), \qquad b_2 = (0, 1, 1), \qquad b_3 = (0, 0, 1).
$$
For $i = 1, \dots, R$, we write $(m_i, n_i, 1)$ for the coordinate of $b_i \in N$ under the basis $\{v_1, v_2, v_3\}$. Assumption \ref{assump:L} implies that $m_i \ge 0$ for all $i$.

Finally, let $f \in \bZ$ be a \emph{framing} on the Aganagic-Vafa brane $L$. This determines a 1-subtorus $T_f := \ker(u_2-fu_1: T' \to \bC^*) \subset T' \subset T$. We take the following notations for the equivariant parameters of the tori:
\begin{align*}
    & R_T := H^*_T(\pt) = \bC[\su_1, \su_2, \su_3], && S_T:= \bC(\su_1, \su_2, \su_3),\\
    & R_{T'} := H^*_{T'}(\pt) = \bC[\su_1, \su_2], && S_{T'}:= \bC(\su_1, \su_2),\\
    & R_{T_f} := H^*_{T_f}(\pt) = \bC[\su_1], && S_{T_f}:= \bC(\su_1).
\end{align*}

\begin{assumption}\label{assump:f} \rm{
We assume that $f \in \bZ$ is generic with respect to $X$, i.e. avoiding a finite subset of $\bZ$ depending on $X$.\footnote{We note in advance that this assumption is needed to ensure that the $T'$-equivariant Poincar\'e pairing and genus-zero Gromov-Witten potential of $X$ have well-defined weight restrictions to $\su_2 - f\su_1 = 0$, to be used in Section \ref{sect:recursive} onwards. This assumption is not required for the open/closed correspondence \cite{LY22} (Theorem \ref{thm:OpenClosed}) and is not the counterpart of \cite[Assumption 3.3]{LY21}.} 
}
\end{assumption}

\subsection{Closed geometry}\label{sec:closed}
Under the open/closed correspondence \cite{Mayr01,LY21,LY22}, the closed geometry corresponding to the open geometry $(X,L,f)$ is a smooth toric Calabi-Yau 4-fold $\tX$ that takes the form
$$
    \tX = \Tot(\cO_{X \sqcup D}(-D)),
$$
where $X \sqcup D$ is a toric partial compactification of $X$ given by adding an additional toric divisor $D$ corresponding to the ray generated by $(-1, -f, 0) \in N$.\footnote{This is the construction in \cite{LY21} and is sufficient for the purpose of this paper. In \cite{LY22}, assuming that $X$ is semi-projective, the corresponding toric 4-fold can be further taken to be a semi-projective partial compactification of $\tX$ which may be an orbifold.} In $X \sqcup D$, the $T$-invariant line $l_{\tau_0} \cong \bC$ that $L$ intersects is compactified by an additional $T$-fixed point into a $\bP^1$ whose normal bundle is isomorphic to $\cO_{\bP^1}(f) \oplus \cO_{\bP^1}(-f-1)$. There is an inclusion
$$
    \iota: X \to X \sqcup D \to \tX.
$$

Let $\tN := N \oplus \bZ \cong \bZ^4$ and $\tT := \tN \otimes \bC^* \cong (\bC^*)^4$. We view $N$ as a sublattice of $\tN$ and let $v_4$ be a generator of the additional $\bZ$-component. The toric geometry of $\tX$ can be described by a fan $\tSi \in \tN_{\bR} := \tN \otimes \bR \cong \bR^4$ as follows. The rays of $\tSi$ are given by
$$
    \tSi(1) = \{\trho_1, \dots, \trho_R, \trho_{R+1}, \trho_{R+2} \}
$$
where under the basis $\{v_1, \dots, v_4\}$ of $\tN$, the primitive generators $\tb_i \in \tN$ of the rays $\trho_i$, $i =1, \dots, R+2$, have the following coordinates:
$$
    \tb_i = (b_i, 0) = (m_i, n_i, 1, 0) \qquad \text{for $i = 1, \dots, R$},
$$
$$
    \tb_{R+1} = (-1, -f, 1, 1), \qquad \tb_{R+2} = (0, 0, 1, 1).
$$
In  $\tX = \Tot(\cO_{X \sqcup D}(-D))$, the toric divisor $V(\trho_{R+1})$ is the restriction of the line bundle $\cO_{X \sqcup D}(-D)$ to $D$ and $V(\trho_{R+2}) = X \sqcup D$ is the zero section.

We describe cones $\tsi$ in $\tSi$ by the index sets
$$
    I'_{\tsi} := \{i \in \{1, \dots, R+2\} : \trho_i \subseteq \tsi\}, \qquad I_{\tsi} := \{1, \dots, R+2\} \setminus I'_{\tsi}.
$$
First, $\tSi$ contains $\Sigma$ as a subfan. Any cone $\sigma \in \Sigma(d)$, $d = 0, \dots, 3$, can be viewed as a cone in $\tSi(d)$ with $I'_{\sigma}$ preserved, and there is a cone $\iota(\sigma) \in \tSi(d+1)$ given by
$$
    I'_{\iota(\sigma)} = I'_{\sigma} \sqcup \{R+2\}.
$$
This induces an injective map $\iota: \Sigma(d) \to \tSi(d+1)$.\footnote{We will abuse notations and use `$\iota$' to denote various inclusions maps.} For maximal cones in $\tSi$, we have
$$
    \tSi(4) = \iota(\Sigma(3)) \sqcup \{\tsi_0\},
$$
where the additional cone $\tsi_0$ is characterized by
$$
    I'_{\tsi_0} = \{2, 3, R+1, R+2\}.
$$
Note that $\tsi_0$ is the only $4$-cone that contains the ray $\trho_{R+1}$. Moreover, the map $\iota: \Sigma(2) \to \tSi(3)$ restricts to an injective map $\iota: \Sigma(2)_c \to \tSi(3)_c$, and we have
$$
    \tSi(3)_c = \iota(\Sigma(2)_c) \sqcup \{\iota(\tau_0)\}.
$$
Indeed, the $\tT$-invariant line $l_{\iota(\tau_0)} \cong \bP^1$ is the compactification of $l_{\tau_0} \cong \bC \subset X$ described at the beginning of this subsection.

Let $\tM := \Hom(\tN, \bZ)$, which is the character lattice of the 4-torus $\tT$, and $\{u_1, \dots, u_4\}$ be the basis of $\tM$ dual to the basis $\{v_1, \dots, v_4\}$ of $\tN$. Here we abuse notations since $u_1, u_2, u_3 \in \tM$ are natural lifts of the corresponding elements of $M$ defined before under the projection $\tM \to M$. We consider a 3-subtorus $\tT':= \ker (u_3: \tT \to \bC^*) \cong (\bC^*)^3$ of $\tT$, which contains $T'$ and $T_f$ as subtori. The fixed points and invariant lines of $\tX$ under the $\tT'$-action are the same as those under the $\tT$-action. We introduce the following notations:
\begin{align*}
    & R_{\tT} := H^*_T(\pt) = \bC[\su_1, \su_2, \su_3, \su_4], && S_{\tT}:= \bC(\su_1, \su_2, \su_3, \su_4),\\
    & R_{\tT'} := H^*_{T'}(\pt) = \bC[\su_1, \su_2, \su_4], && S_{\tT'}:= \bC(\su_1, \su_2, \su_4).
\end{align*}

\subsection{Second homology and effective curve classes}\label{sect:CurveClass}
The intersection of $L$ with $l_{\tau_0} \cong \bC$ in $X$ is isomorphic to $S^1$ and bounds a holomorphic disk $B$ in $l_{\tau_0}$, oriented by the holomorphic structure of $X$. The disk $B$ represents a class $[B]$ in $H_2(X, L; \bZ)$, and its boundary $\partial B = L \cap l_{\tau_0}$ generates $H_1(L; \bZ) \cong \bZ[\partial B]$. We have a splitting
$$
    H_2(X, L; \bZ) \cong H_2(X; \bZ) \oplus \bZ [B].
$$

We introduce the following notations for the semigroups of effective classes:
\begin{equation}\label{eqn:EffClasses}
    \begin{aligned}
        & E(X) := \NE(X) \cap H_2(X; \bZ),\\
        & E(X, L) := E(X) \oplus \bZ_{\ge 0} [B] \subset H_2(X, L; \bZ),\\
        & E(\tX) := \NE(\tX) \cap H_2(\tX; \bZ).
    \end{aligned}
\end{equation}
The inclusion $\iota: X \to \tX$ induces an isomorphism
$$
    \iota_*: H_2(X, L; \bZ) \to H_2(\tX; \bZ), \qquad \beta + d[B] \mapsto \iota_*(\beta) + d[l_{\iota(\tau_0)}]
$$
which restricts to a semigroup isomorphism
$$
    \iota_*: E(X,L) \cong E(\tX). 
$$
We will thus use the coordinates $(\beta, d) \in E(X) \oplus \bZ_{\ge 0}$ for both semigroups above. The pairing between $\tbeta = (\beta, d) \in E(\tX)$ and the divisor class $[V(\trho_{R+1})]$ is
$$
    \tbeta \cdot [V(\trho_{R+1})] = d.
$$

\subsection{Flags and tangent weights at torus-fixed points}
For a flag $(\tau, \sigma) \in F(\Sigma)$, let
$$
    \bw(\tau, \sigma) := c_1^{T'}(T_{p_\sigma}l_{\tau}) \in H^2_{T'}(\pt; \bZ)
$$
be the weight of the $T'$-action on tangent space $T_{p_\sigma}l_{\tau}$ of $l_{\tau}$ at $p_{\sigma}$. Similarly, for a flag $(\ttau, \tsi) \in F(\tSi)$, let
$$
    \tbw(\ttau, \tsi) := c_1^{\tT'}(T_{p_{\tsi}}l_{\ttau}) \in H^2_{\tT'}(\pt; \bZ).
$$

The maps $\iota: \Sigma(d) \to \tSi(d+1)$ defined in Section \ref{sec:closed} induce an injective map $\iota: F(\Sigma) \to F(\tSi)$, $(\tau, \sigma) \mapsto (\iota(\tau), \iota(\sigma))$. We have
$$
    \tbw(\iota(\tau), \iota(\sigma)) \big|_{\su_4 = 0} = \bw(\tau, \sigma).
$$
Each 4-cone $\iota(\sigma) \in \tSi(4)$ with $\sigma \in \Sigma(3) \subset \tSi(3)$ belongs to an additional flag $(\sigma, \iota(\sigma)) \in F(\tSi)$. We have
$$
    \tbw(\sigma, \iota(\sigma)) = \su_4.
$$

The additional 4-cone $\tsi_0 \in \tSi(4) \setminus \iota(\Sigma(3))$ belongs to the flags
$$
    (\iota(\tau_0), \tsi_0), (\ttau_2, \tsi_0), (\ttau_3, \tsi_0), (\ttau_4, \tsi_0) \in F(\tSi)
$$
where the facets $\ttau_2, \ttau_3, \ttau_4$ of $\tsi_0$ are given by
$$
    I'_{\ttau_2} = \{3, R+1, R+2\}, \quad I'_{\ttau_3} = \{2, R+1, R+2\}, \quad I'_{\ttau_4} = \{2, 3, R+1\}.
$$
The tangent weights are given by
$$
    \tbw(\iota(\tau_0), \tsi_0) = -\su_1, \quad \tbw(\ttau_2, \tsi_0) = -f\su_1 + \su_2, \quad \tbw(\ttau_3, \tsi_0) = f\su_1 - \su_2 - \su_4, \quad \tbw(\ttau_4, \tsi_0) = \su_1 + \su_4.
$$

\subsection{Equivariant cohomology and bases}\label{sect:EquivCohomology}


We fix an ordering of the $T'$-fixed points of $X$ by
$$
    p_1, \dots, p_m
$$
and denote the corresponding $\tT'$-fixed points of $\tX$ by 
$$
    \tp_1, \dots, \tp_m.
$$
We denote the additional $\tT'$-fixed point $p_{\tsi_0}$ of $\tX$ by $\tp_{m+1}$. 

We consider the basis $\{\phi_1,\dots,\phi_m\}$ of $H^*_{T'}(X) \otimes_{R_{T'}} S_{T'}$ defined as
$$
    \phi_i := \frac{[p_i]}{e_{T'}(T_{p_i}X)}=\frac{[p_i]}{\Delta^{i,T'}},\qquad \Delta^{i,T'} := e_{T'}(T_{p_i}X).
$$
Then for $i, j = 1, \dots, m$, we have
$$
    \phi_i \cup \phi_j=\delta_{ij}\phi_i, \qquad (\phi_i,\phi_j)_{X,T'}=\frac{\delta_{ij}}{\Delta^{i,T'}}
$$
where $(-,-)_{X, T'}$ is the $T'$-equivariant Poincar\'e pairing on $X$. It follows that $\{\phi_1,\dots,\phi_m\}$ is a canonical basis of the semi-simple Frobenius algebra
$$
    (H^*_{T'}(X) \otimes_{R_{T'}} S_{T'},\cup,(-,-)_{X,T'}).
$$

Similarly, we define the basis $\{\tphi_1,\dots,\tphi_m,\tphi_{m+1} \}$ of $H^*_{\tT'}(\tX) \otimes_{R_{\tT'}} S_{\tT'}$ as
$$
    \tphi_i := \frac{[\tp_i]}{e_{\tT'}(T_{\tp_i}\tX)} = \frac{[\tp_i]}{\Delta^{i,\tT'}},\qquad \Delta^{i,\tT'} := e_{\tT'}(T_{\tp_i}\tX).
$$
Note that for any $i, j = 1, \dots, m$, we have
\begin{equation}\label{eqn:FixedPtRestrict}
    \tphi_i \big|_{\tp_j} = \phi_i \big|_{p_j} = \delta_{ij}, \qquad \tphi_i \big|_{\tp_{m+1}} = 0,
\end{equation}
and
$$
    \su_4^{-1} \Delta^{i,\tT'} \big|_{\su_4 = 0} = \Delta^{i,T'}.
$$
For $i, j = 1, \dots, m+1$, we have
$$
    \tphi_i \cup \tphi_j = \delta_{ij}\tphi_i,\qquad (\tphi_i,\tphi_j)_{\tX,\tT'} = \frac{\delta_{ij}}{\Delta^{i,\tT'}}
$$
where $(-,-)_{\tX, \tT'}$ is the $\tT'$-equivariant Poincar\'e pairing on $\tX$. It follows that $\{\tphi_1,\dots,\tphi_{m+1}\}$ is a canonical basis of the semi-simple Frobenius algebra
\begin{equation}\label{eqn:tXClassical}
    (H^*_{\tT'}(\tX) \otimes_{R_{\tT'}} S_{\tT'},\cup,(-,-)_{\tX,\tT'}).
\end{equation}

Moreover, for $i = 1, \dots, R+2$, let
$$
    \tD^{\tT'}_i := c_1^{\tT'}(\cO_{\tX}(V(\trho_{i}))) \in H^2_{\tT'}(\tX)
$$
denote the $\tT'$-equivariant Poincar\'e dual of the divisor $V(\trho_i)$. Specifically we denote
$$
    \tD := \tD^{\tT'}_{R+1}.
$$
Since the divisor $V(\trho_{R+1})$ only contains the $\tT'$-fixed point $p_{\tsi_0} = \tp_{m+1}$, we have that
\begin{equation}\label{eqn:DivPtClasses}
    \tD = \tD\big|_{\tp_{m+1}} \tphi_{m+1} = -\su_1 \tphi_{m+1}.
\end{equation}

\section{Gromov-Witten theory and open/closed correspondence}\label{sect:GW}
In this section, we review the different types of Gromov-Witten invariants involved in our study of Frobenius structures, specifically the closed invariants of $X$ and $\tX$ as well as the open invariants of $(X, L)$. We then use the open/closed correspondence \cite{LY21,LY22} to obtain a refined relation among the generating functions of Gromov-Witten invariants (Theorem \ref{thm:OpenClosed}).

\subsection{Closed Gromov-Witten invariants of $X$ and $\tX$}\label{sect:ClosedGW}
We refer to \cite{Liu13} for additional details on virtual localization \cite{GP99} in the Gromov-Witten theory of toric varieties.

For $n \in \bZ_{\ge 0}$ and effective class $\beta \in E(X)$ (see \eqref{eqn:EffClasses}), let $\Mbar_{0,n}(X, \beta)$ be the moduli space of genus-zero, $n$-pointed, degree-$\beta$ stable maps to $X$. Given $T'$-equivariant cohomology classes $\gamma_1, \dots, \gamma_n \in H^*_{T'}(X) \otimes_{R_{T'}} S_{T'}$ as insertions, we define the \emph{closed Gromov-Witten invariant}
$$
    \inner{\gamma_1, \dots, \gamma_n}^{X, T'}_{0,n,\beta} := \int_{[\Mbar_{0,n}(X, \beta)^{T'}]^{\vir}} \frac{\prod_{i=1}^n \ev_i^*(\gamma_i)}{e_{T'}(N^{\vir})} \qquad \in S_{T'}
$$
by localization with respect to the torus $T'$, where for $i = 1, \dots, n$, $\ev_i: \Mbar_{0,n}(X, \beta) \to X$ is the evaluation map at the $i$-th marked point.

We now define a generating function of such invariants. The \emph{Novikov ring} of $X$ is the completion of the semigroup ring of $E(X)$,
$$
    \Lambda_X := \left\{\sum_{\beta \in E(X)} c_\beta Q^\beta : c_\beta \in \bC \right\}.
$$
in which we use $Q^\beta$ to denote the semigroup ring element corresponding to $\beta \in E(X)$. We will also use the equivariant versions
$$
    \Lambda_X^{T'} := S_{T'} \otimes_{\bC} \Lambda_X, \qquad \Lambda_X^{T_f} := S_{T_f} \otimes_{\bC} \Lambda_X.
$$
Consider the basis $\{\phi_1,\dots,\phi_m\}$ of $H^*_{T'}(X) \otimes_{R_{T'}} S_{T'}$ defined in Section \ref{sect:EquivCohomology}. Let
$$
    t := \sum_{i=1}^m t^i \phi_i
$$
where $t^1, \dots, t^m$ are formal variables viewed as coordinates. The genus-zero, $T'$-equivariant \emph{Gromov-Witten potential} of $X$ is the following generating function of closed Gromov-Witten invariants:
$$
    F_0^{X,T'}(t^1, \dots, t^m) := \sum_{\beta \in E(X)} \sum_{n \in \bZ_{\ge 0}} \frac{\inner{t, \dots, t}^{X, T'}_{0,n,\beta}}{n!} Q^\beta \qquad \in \Lambda_X^{T'}\formal{t^1, \dots, t^m}.
$$

Now we set up a parallel theory for $\tX$. For $n \in \bZ_{\ge 0}$ and effective class $\tbeta \in E(\tX)$ (see \eqref{eqn:EffClasses}), let $\Mbar_{0,n}(\tX, \tbeta)$ be the moduli space of genus-zero, $n$-pointed, degree-$\tbeta$ stable maps to $\tX$. Given $\tT'$-equivariant cohomology classes $\tgamma_1, \dots, \tgamma_n \in H^*_{\tT'}(\tX) \otimes_{R_{\tT'}} S_{\tT'}$ as insertions, we define the \emph{closed Gromov-Witten invariant}
$$
    \inner{\tgamma_1, \dots, \tgamma_n}^{\tX, \tT'}_{0,n,\tbeta} := \int_{[\Mbar_{0,n}(\tX, \tbeta)^{\tT'}]^{\vir}} \frac{\prod_{i=1}^n \ev_i^*(\tgamma_i)}{e_{\tT'}(N^{\vir})} \qquad \in S_{\tT'}
$$
by localization with respect to the torus $\tT'$, where for $i = 1, \dots, n$, $\ev_i: \Mbar_{0,n}(\tX, \tbeta) \to \tX$ is the evaluation map at the $i$-th marked point.

The \emph{Novikov ring} of $\tX$ is the completion of the semigroup ring of $E(\tX)$:
$$
    \Lambda_{\tX} := \left\{\sum_{\tbeta \in E(\tX)} c_{\tbeta} \tQ^{\tbeta} : c_{\tbeta} \in \bC \right\}.
$$
in which we use $\tQ^{\tbeta}$ to denote the semigroup ring element corresponding to $\tbeta \in E(\tX)$. We will also use the equivariant version
$$
    \Lambda_{\tX}^{\tT'} := S_{\tT'} \otimes_{\bC} \Lambda_{\tX}.
$$
Consider the basis $\{\tphi_1,\dots,\tphi_m, \tphi_{m+1}\}$ of $H^*_{\tT'}(\tX) \otimes_{R_{\tT'}} S_{\tT'}$ defined in Section \ref{sect:EquivCohomology}. Let
$$
    \tilde{t} := \sum_{i=1}^m t^i \tphi_i, \qquad \hat{t} := \tilde{t} + t^{m+1}\tphi_{m+1}
$$
where $t^1, \dots, t^m$ are formal variables as before and $t^{m+1}$ is an additional formal variable. The genus-zero, $\tT'$-equivariant \emph{Gromov-Witten potential} of $\tX$ is the following generating functions of closed Gromov-Witten invariants:
$$
    F_0^{\tX,\tT'}(t^1, \dots, t^m, t^{m+1}) := \sum_{\tbeta \in E(\tX)} \sum_{n \in \bZ_{\ge 0}} \frac{\inner{\hat{t}, \dots, \hat{t}}^{\tX, \tT'}_{0,n,\tbeta}}{n!}\tQ^{\tbeta} \qquad \in \Lambda_{\tX}^{\tT'}\formal{t^1, \dots, t^m, t^{m+1}}.
$$

By \eqref{eqn:DivPtClasses}, we have
$$
    \hat{t} = \tilde{t} - \frac{t^{m+1}}{\su_1}\tD.
$$
Recall from Section \ref{sect:CurveClass} that each $\tbeta \in E(\tX)$ can be uniquely expressed as $\iota_*(\beta) + d[l_{\iota(\tau_0)}]$ for some $\beta \in E(X)$ and $d \in \bZ_{\ge 0}$. The divisor equation then implies that
\begin{equation}\label{eqn:4FoldF0}
    F_0^{\tX,\tT'}(t^1, \dots, t^m, t^{m+1}) = \frac{(t^{m+1})^3}{6\Delta^{m+1,\tT'}} + \sum_{\tbeta = (\beta, d) \in E(\tX)} \sum_{n \in \bZ_{\ge 0}} \frac{\inner{\tilde{t}, \dots, \tilde{t}}^{\tX, \tT'}_{0,n,\tbeta}}{n!}\tQ^{\iota_*(\beta)}\left(e^{-\frac{t^{m+1}}{\su_1}} \tQ^{[l_{\iota(\tau_0)}]} \right)^d.
\end{equation}
Here the term $\frac{(t^{m+1})^3}{6\Delta^{m+1,\tT'}}$ captures the $t^{m+1}$-dependence of the (3-pointed) degree-$0$ invariants in $F_0^{\tX,\tT'}$:
$$
    \frac{\inner{t^{m+1}\tphi_{m+1}, t^{m+1}\tphi_{m+1}, t^{m+1}\tphi_{m+1}}^{\tX, \tT'}_{0,3,0}}{3!} = \frac{(t^{m+1})^3}{6} (\tphi_{m+1} \cup \tphi_{m+1},\tphi_{m+1})_{\tX,\tT'} = \frac{(t^{m+1})^3}{6\Delta^{m+1,\tT'}}.
$$
Note that $\tphi_i \cup \tphi_{m+1} = 0$ for any $i = 1, \dots, m$.

\subsection{Open Gromov-Witten invariants of $(X, L, f)$}\label{sect:OpenGW}
Recall from Section \ref{sect:OpenGeometry} that the $T'_{\bR}$-action on $X$ preserves the Lagrangian $L$ and may thus be used to define open Gromov-Witten invariants, specifically disk invariants which are virtual counts of open stable maps from genus-zero domains with one boundary component. We now recall the definitions and refer to \cite{FL13,FLT22} for additional details.

For $n \in \bZ_{\ge 0}$ and effective class $\beta' = (\beta, d) \in E(X, L)$ (see \eqref{eqn:EffClasses}) with $d \in \bZ_{>0}$, let $\Mbar_{(0,1),n}(X,L \mid \beta', d)$ be the moduli space of degree-$\beta'$ stable maps to $(X,L)$ from domains $(C, \partial C)$ with
\begin{itemize}
    \item topological type $(0,1)$, i.e. $C$ is a nodal Riemann surface of arithmetic genus zero with one open disk removed, and
    \item $n$ interior marked points disjoint from $\partial C$.
\end{itemize}
Given $T'$-equivariant (or equivalently $T'_{\bR}$-equivariant) cohomology classes $\gamma_1, \dots, \gamma_n \in H^*_{T'}(X) \otimes_{R_{T'}} S_{T'}$ as insertions, we define the \emph{disk invariant}
$$
    \inner{\gamma_1, \dots, \gamma_n}^{X, L}_{(0,1),n,\beta', d} := \int_{[\Mbar_{(0,1),n}(X,L \mid \beta', d)^{T'_{\bR}}]^{\vir}} \frac{\prod_{i=1}^n \ev_i^*(\gamma_i)}{e_{T'_{\bR}}(N^{\vir})} \qquad \in S_{T'}
$$
by localization with respect to the compact torus $T'_{\bR}$, where for $i = 1, \dots, n$, $\ev_i: \Mbar_{(0,1),n}(X,L \mid \beta', d) \to X$ is the evaluation map at the $i$-th marked point. Here, we identify the field of fractions of $H^*_{T'_{\bR}}(\pt)$ with $S_{T'}$. Furthermore, using the framing $f \in \bZ$, we take a weight restriction to define
$$
    \inner{\gamma_1, \dots, \gamma_n}^{X, (L,f)}_{(0,1),n,\beta', d} := \inner{\gamma_1, \dots, \gamma_n}^{X, L}_{(0,1),n,\beta', d} \big|_{\su_2 - f\su_1 = 0} \qquad \in S_{T_f}.
$$
In this paper, we will only need to work with insertions for which the above weight restriction of the disk invariant is defined.

The completion of the semigroup ring of $E(X,L)$ is
$$
    \Lambda_{X,L}:= \left\{\sum_{(\beta,d) \in E(X,L)} c_{(\beta, d)} Q^\beta\sX_0^d : c_{(\beta, d)} \in \bC \right\} = \Lambda_X \formal{\sX_0}
$$
in which we introduce the new formal variable $\sX_0$ for the last component. Note that the isomorphism $\iota_*:E(X,L) \cong E(\tX)$ induces an isomorphism $\Lambda_{X,L} \cong \Lambda_{\tX}$ under the change of variables $\tQ^{\iota_*(\beta)} = Q^\beta$, $\tQ^{[l_{\iota(\tau_0)}]} = \sX_0$. We will also use the equivariant version
$$
    \Lambda_{X,L}^{T_f} := S_{T_f} \otimes_{\bC} \Lambda_{X,L} = \Lambda_X^{T_f} \formal{\sX_0}.
$$
Consider the basis $\{\phi_1,\dots,\phi_m\}$ of $H^*_{T'}(X) \otimes_{R_{T'}} S_{T'}$ and $t = \sum_{i=1}^m t^i \phi_i$ as in Section \ref{sect:ClosedGW}. Let $t^o$ be an additional formal variable for the open sector. The $T_f$-equivariant \emph{disk potential} of $(X,L,f)$ is the following generating functions of disk invariants:
\begin{align*}
    F_{0,1}^{X,(L, f)}(t^1, \dots, t^m, t^o) :=& \sum_{\substack{(\beta, d) \in E(X,L) \\ d \in \bZ_{>0}}} \sum_{n \in \bZ_{\ge 0}} \frac{\inner{t, \dots, t}^{X, (L,f)}_{(0,1),n,\beta+d[B], d}}{n!}Q^\beta(e^{t^o}\sX_0)^d \\
    & \in \Lambda_{X}^{T_f}\formal{t^1, \dots, t^m, e^{t^o}\sX_0}.
\end{align*}
Conceptually, we may view $t^o$ as parameterizing a divisor-like insertion arising from the open sector and $\sX := e^{t^o}\sX_0$ as parameterizing the winding numbers of disk invariants. Note that $F_{0,1}^{X,(L, f)}$ is supported on the ideal of $\Lambda_{X,L}$ generated by $\sX_0$. For later use, we introduce the following modified version:
\begin{align*}
    \int F_{0,1}^{X,(L, f)}(t^1, \dots, t^m, t^o) := & \sum_{\substack{(\beta, d) \in E(X,L) \\ d \in \bZ_{>0}}} \sum_{n \in \bZ_{\ge 0}} \frac{\inner{t, \dots, t}^{X, (L,f)}_{(0,1),n,\beta+d[B], d}}{d \cdot n!}Q^\beta(e^{t^o}\sX_0)^d\\
    & \in \Lambda_{X}^{T_f}\formal{t^1, \dots, t^m, e^{t^o}\sX_0}
\end{align*}
where $\int$ is interpreted as taking the antiderivative with respect to $t^o$. We note that the insertions $\phi_1, \dots, \phi_m$ are homogeneous of degree 0 and do not introduce additional poles along $\su_2-f\su_1$. Thus, the weight restriction to $\su_2 - f\su_1 = 0$ in the definition of the disks invariants in $F_{0,1}^{X,(L, f)}$ is valid. Similarly, it is valid to apply this weight restriction to the closed invariants of $X$ in $F_0^{X,T'}$.


\subsection{Open/closed correspondence}\label{sect:OpenClosed}
The open/closed correspondence \cite{LY21,LY22} identifies the genus-zero open Gromov-Witten theory of $(X,L,f)$ and closed Gromov-Witten theory of $\tX$ at the numerical level of invariants as well as the level of generating functions. In this paper, we use the following statement of the correspondence. We introduce the notation
\begin{equation}\label{eqn:sv}
    \sv := \begin{cases}
        \tbw(\ttau_2, \tsi_0) = \su_2 - f\su_1 & \text{if $f \ge 0$},\\
        -\tbw(\ttau_3, \tsi_0) = \su_2 - f\su_1 + \su_4 & \text{if $f < 0$}.
    \end{cases}
\end{equation}

\begin{theorem}[\cite{LY22}]\label{thm:OpenClosed}
The Gromov-Witten potential $F_0^{\tX,\tT'}$ of $\tX$ can be expanded as
\begin{equation}\label{eqn:F0Expansion}
\begin{aligned}
    F_0^{\tX,\tT'}(t^1, \dots, t^m, t^{m+1}) = & \frac{(t^{m+1})^3}{6\Delta^{m+1,\tT'}} + \su_4^{-1}\tA(t^1, \dots, t^m) + \sv^{-1}\tB(t^1, \dots, t^m, t^{m+1}) \\
    & + \su_4\sv^{-1}\tC_1 (t^1, \dots, t^m, t^{m+1}) + \tC_2(t^1, \dots, t^m, t^{m+1})
\end{aligned}
\end{equation}
where
\begin{enumerate}[label=(\alph*)]
    \item Each of $\tA, \tB, \tC_1, \tC_2$ has a well-defined weight restriction to $\su_4 = 0,  \su_2 - f\su_1 = 0$.
    
    \item $\tA$ is supported on the Novikov variables $\{\tQ^{\iota_*(\beta)}: \beta \in E(X)\}$ and
    $$
        \tA(t^1, \dots, t^m) \big|_{\su_4 = 0} = F_0^{X,T'}(t^1, \dots, t^m)
    $$
    after the change of variables $\tQ^{\iota_*(\beta)} = Q^\beta$.

    \item We have
    $$
        \tB(t^1, \dots, t^m, t^{m+1}) \big|_{\su_4 = 0, \su_2 - f\su_1 = 0} = \int F_{0,1}^{X,(L, f)}(t^1, \dots, t^m, t^o)
    $$
    after the change of variables $\tQ^{\iota_*(\beta)} = Q^\beta$, $\tQ^{[l_{\iota(\tau_0)}]} = \sX_0$, and $t^{m+1} = -\su_1 t^o$.
\end{enumerate}
\end{theorem}

The statement of Theorem \ref{thm:OpenClosed} differs from the results in \cite{LY22}, particularly Theorems 4.1 and 5.4 there, in that it uses the classes $\phi_1,\dots,\phi_m$ and their counterparts $\tphi_1,\dots,\tphi_m$ to parameterize insertions, and that it also involves closed Gromov-Witten invariants of $X$. Nevertheless, it directly follows from the localization analysis and vanishing arguments in \cite[Section 4]{LY22}. We defer the derivation details to Appendix \ref{appdx:OpenClosed}.

\section{Frobenius structures on closed Gromov-Witten theory}\label{sect:closedFrob}
In this section, we review the equivariant formal Frobenius structures determined by the closed Gromov-Witten theory of $X$, $\tX$ and specifically the WDVV equations. Under the open/closed correspondence, we use the WDVV equation for $\tX$ to deduce a collection of non-linear partial differential equations that involve the open and closed Gromov-Witten invariants of $(X,L)$ (Proposition \ref{prop:openWDVV}).

\subsection{Formal Frobenius and $F$-manifolds}\label{sect:FormalDef}
We first recall the definition of formal Frobenius manifolds over a general base ring $R$ which is a commutative algebra over $\bC$, extending Definition \ref{def:Frobenius}. We refer to \cite[Chapter 2]{LP} for additional details.


\begin{definition}\label{def:FormalFrobGeneralBase}
A \emph{formal Frobenius manifold over $R$} consists of the data $(\hM,g,A,\one)$ where:
\begin{enumerate}
\item $\hM = \Spec (R\formal{K^\vee})$ is a formal manifold over $R$ defined by the completion of a free $R$-module $K$ of rank $m$ at the origin, where $K^\vee := \Hom_R(K,R)$;
\item $g$ is a formal, flat, $R$-linear, symmetric, nondegenerate quadratic form on the formal tangent bundle $\cT_{\hM}$ over $R$;
\item $A$ is a formal, $R$-linear, symmetric tensor
$$
    A:\cT_{\hM}\otimes \cT_{\hM}\otimes \cT_{\hM} \to \cO_{\hM}.
$$
\item $\one$ is a formal vector field on $\hM$ over $R$.
\end{enumerate}
The above data are required to satisfy the potentiality, associativity, and unit conditions as in Definition \ref{def:Frobenius}.
\end{definition}

A formal Frobenius manifold $\hM$ over $R$ may alternatively be viewed as a relative formal complex Frobenius manifold over the affine base $\Spec (R)$. Elements in $R$ pull back to constants in the structure sheaf $\cO_{\hM}$.

Given a formal Frobenius manifold $\hM = \Spec (R\formal{K^\vee})$ as above, the origin is the only point in $\hM$ and $\cT_{\hM} \cong K \otimes_R \cO_{\hM}$. The product $\star$ defined by the associativity condition specializes to an $R$-algebra $(K, \star)$ at the origin.

\begin{definition}
A formal Frobenius manifold $\hM$ over $R$ is \emph{semi-simple} if the induced $R$-algebra $(K \otimes_R R\formal{K^\vee}, \star)$ is isomorphic to $\bigoplus_{i = 1}^m R\formal{K^\vee}$ with the product algebra structure.
\end{definition}

Similarly, we may define flat formal $F$-manifolds over the general base ring $R$, extending Definition \ref{def:F-mfd}.

\begin{definition}\label{def:FormalFGeneralBase}
A \emph{flat formal $F$-manifold over $R$} consists of the data $(\hM,\nabla,\star,\one)$ where
\begin{enumerate}
\item $\hM = \Spec (R\formal{K^\vee})$ is a formal manifold over $R$ defined by the completion of a free $R$-module $K$ of rank $m+1$ at the origin;
\item $\nabla$ is an $R$-linear connection on the formal tangent bundle $\cT_{\hM}$,
\item $\star$ defines an algebra structure on $\cT_{\hM}$,
\item $\one$ is a $\nabla$-flat formal vector field on $\hM$ over $R$ which is a unit for $\star$.
\end{enumerate}
The above data satisfy the condition that the connection $\nabla^z:=\nabla-\frac{1}{z}\star$ is flat and symmetric for any $z\in\bP^1$.
\end{definition}

\subsection{Gromov-Witten case}\label{sect:generalFrob}
Let $\cX$ be a smooth projective variety. Let $\{T_i\}_{i=1}^m$ be a basis of $H^*(\cX)$ and $t^1,\dots,t^m$ be the corresponding coordinates. Consider the genus-zero Gromov-Witten potential $F_0^{\cX}$ of $\cX$.
Let
$$
g_{ij}=( T_i,T_j  )_{\cX}=\int_{\cX}T_i\cup T_j
$$
and $(g^{ij})=(g_{ij})^{-1}$.

Let $\partial_i := \frac{\partial}{\partial t^i}$. As stated in Theorem \ref{thm:WDVV general}, for any $i,j,k,l\in\{1,\dots,m\}$, the following WDVV equation holds:
$$
\partial_i\partial_j\partial_\nu F_0^{\cX}\cdot g^{\nu\mu}\cdot \partial_\mu\partial_k\partial_l F_0^{\cX}
=\partial_j\partial_k\partial_\nu F_0^{\cX}\cdot g^{\nu\mu}\cdot \partial_\mu\partial_i\partial_l F_0^{\cX}.
$$
For any $i,j\in\{1,\dots,m\}$, we define the quantum product $T_i\star_t T_j$ as
$$
(T_i\star_t T_j,T_k)_{\cX}=\frac{\partial^3 F_0^{\cX}}{\partial t^i\partial t^j\partial t^k}.
$$
The WDVV equation implies that the quantum product $\star_t$ is associative.

Moreover, we can define a formal Frobenius manifold as follows. Let
$$
H := \mathrm{Spec}(\Lambda_{\cX} [ t^1,\dots,t^m]),
$$
where $\Lambda_{\cX}$ is the Novikov ring of $\cX$. Let $\hH$ be the formal completion of $H$ along the origin:
$$
\hH := \mathrm{Spec}(\Lambda_{\cX}\formal{t^1,\dots,t^m}).
$$
Let $\cO_{\hH}$ be the structure sheaf on $\hH$ and $\cT_{\hH}$ be the tangent sheaf on
$\hH$.
Then $\cT_{\hH}$ is a sheaf of free $\cO_{\hH}$-modules of rank $N$.
Given an open set $U$ in $\hH$, we have
$$
\cT_{\hH}(U)  \cong \bigoplus_{i=1}^m \cO_{\hH}(U) \frac{\partial}{\partial t^{i}}.
$$
The quantum product and the Poincar\'{e} pairing define the structure of a formal
Frobenius manifold on $\hH$ over $\Lambda_{\cX}$:
$$
\left(\frac{\partial}{\partial t^i}\star_t \frac{\partial}{\partial t^j},\frac{\partial}{\partial t^k}\right)_{\cX}=\frac{\partial^3 F_0^{\cX}}{\partial t^i\partial t^j\partial t^k}, \qquad
\left(\frac{\partial}{\partial t^i}, \frac{\partial}{\partial t^j}\right)_{\cX}=g_{ij}.
$$

The generalization of the WDVV equation to the equivariant setting is straightforward. Suppose $\cX$ admits an action of a torus $\bT$ and let $\{T_i\}_{i=1}^m$ be a basis of $H^*_{\bT}(\cX)$. One only needs to replace $F_0^{\cX}$ by the genus-zero $\bT$-equivariant Gromov-Witten potential $F_0^{\cX,\bT}$ and replace $( T_i,T_j)_{\cX}$ by the $\bT$-equivariant Poincar\'{e} pairing $( T_i,T_j)_{\cX,\bT}$. Then the WDVV equation (Theorem \ref{thm:WDVV general}) still holds. Moreover, in the equivariant setting, $\cX$ can be allowed to be non-compact. We only need $\Mbar_{g,n}(\cX,\beta)^\bT$ to be compact in order to define $\bT$-equivariant Gromov-Witten invariants of $\cX$.

In the equivariant setting, we can still define a formal Frobenius manifold as follows. Let
$$
H := \mathrm{Spec}(\Lambda_{\cX}^{\bT} [ t^1,\dots,t^m]),
$$
where $\Lambda_{\cX}^{\bT}$ is the base change of $\Lambda_{\cX}$ to adjoin equivariant parameters of $\bT$. Let $\hH$ be the formal completion of $H$ along the origin:
$$
\hH :=\mathrm{Spec}(\Lambda_{\cX}^{\bT}\formal{t^1,\dots,t^m}).
$$
Let $\cO_{\hH}$ be the structure sheaf on $\hH$ and $\cT_{\hH}$ be the tangent sheaf on
$\hH$.
Then $\cT_{\hH}$ is a sheaf of free $\cO_{\hH}$-modules of rank $m$.
Given an open set $U$ in $\hH$, we have
$$
\cT_{\hH}(U)  \cong \bigoplus_{i=1}^m \cO_{\hH}(U) \frac{\partial}{\partial t^{i}}.
$$
The quantum product and the $\bT$-equivariant Poincar\'{e} pairing define the structure of a formal Frobenius manifold on $\hH$ over $\Lambda_{\cX}^{\bT}$:
$$
\left(\frac{\partial}{\partial t^i}\star_t \frac{\partial}{\partial t^j},\frac{\partial}{\partial t^k}\right)_{\cX,\bT}=\frac{\partial^3 F_0^{\cX,\bT}}{\partial t^i\partial t^j\partial t^k}, \qquad
\left(\frac{\partial}{\partial t^i}, \frac{\partial}{\partial t^j}\right)_{\cX,\bT}=g_{ij}.
$$

\subsection{Specializing to $X$ and $\tX$}\label{sect:specialFrob}
Now we specialize to the toric Calabi-Yau 3-fold $X$ and the toric Calabi-Yau 4-fold $\tX$. Recall from Section \ref{sect:EquivCohomology} that we defined the bases $\{\phi_1, \dots, \phi_m\}$, $\{\tphi_1, \dots, \tphi_{m+1}\}$ of $H^*_{T'}(X)$, $H^*_{\tT'}(\tX)$ respectively. Let
$$
    \hH_X := \Spec(\Lambda_X^{T'}\formal{t^1, \dots, t^m}), \qquad \hH_{\tX} := \Spec(\Lambda_{\tX}^{\tT'}\formal{t^1, \dots, t^{m+1}})
$$
be the induced equivariant formal Frobenius manifolds constructed as in Section \ref{sect:generalFrob}. The quantum products are given by the closed Gromov-Witten potentials $F_0^{X,T'}, F_0^{\tX,\tT'}$ respectively. The equivariant Poincar\'e parings are \emph{diagonal}:
\begin{align*}
    &g_{ij} := (\phi_i,\phi_j)_{X,T'}=\frac{\delta_{ij}}{\Delta^{i,T'}}, && i,j \in \{1, \dots, m\};\\
    &\tg_{ij} := (\tphi_i,\tphi_j)_{\tX,\tT'}=\frac{\delta_{ij}}{\Delta^{i,\tT'}}, && i,j \in \{1, \dots, m+1\}.
\end{align*}
Note that for $i = 1, \dots, m$ we have
$$
    \tg_{ii} = \frac{1}{\su_4}g_{ii}.
$$

Let $(g^{ij})=(g_{ij})^{-1}$ and $(\tg^{ij})=(\tg_{ij})^{-1}$. For any $i,j,k,l \in \{1, \dots, m\}$, the WDVV equation for $X$ reads
\begin{equation}\label{eqn:X-WDVV}
\partial_i\partial_j\partial_\nu F_0^{X,T'}\cdot g^{\nu\nu}\cdot \partial_\nu\partial_k\partial_l F_0^{X,T'}
=\partial_j\partial_k\partial_\nu F_0^{X,T'}\cdot g^{\nu\nu}\cdot \partial_\nu\partial_i\partial_l F_0^{X,T'}
\end{equation}
where the summation index $\nu$ runs through $1, \dots, m$. For any $i,j,k,l \in \{1, \dots, m+1\}$, the WDVV equation for $\tX$ reads
\begin{equation}\label{eqn:tX-WDVV}
\partial_i\partial_j\partial_\nu F_0^{\tX,\tT'}\cdot \tg^{\nu\nu}\cdot \partial_\nu\partial_k\partial_l F_0^{\tX,\tT'}
=\partial_j\partial_k\partial_\nu F_0^{\tX,\tT'}\cdot \tg^{\nu\nu}\cdot \partial_\nu\partial_i\partial_l F_0^{\tX,\tT'}
\end{equation}
where the summation index $\nu$ runs through $1, \dots, m+1$.

\subsection{Recursive relations for open and closed invariants}\label{sect:recursive}
Now we combine the WDVV equation \eqref{eqn:tX-WDVV} for $\tX$ and the open/closed correspondence (Theorem \ref{thm:OpenClosed}) to obtain the following non-linear partial differential equations for the closed Gromov-Witten potential $F_0^{X,T'}$ of $X$ and the disk potential $F_{0,1}^{X,(L, f)}$ of $(X, L, f)$. For $i,j \in \{1, \dots, m\}$, we set
\begin{equation}\label{eqn:PairingRestrict}
    h_{ij} := g_{ij} \big|_{\su_2 - f\su_1 = 0}
\end{equation}
which is well-defined by Assumption \ref{assump:f}. Let $(h^{ij})=(h_{ij})^{-1}$.

\begin{proposition}\label{prop:openWDVV}
Denote $\partial_o := \frac{\partial}{\partial t^o}$. We have
\begin{enumerate}[label=(\Roman*), align=left, leftmargin=0pt, labelwidth=0pt, itemindent=!]
    \item For $i, j, k, l \in \{1, \dots, m\}$:
        \begin{enumerate}[label=(I\alph*), align=left, leftmargin=0pt, labelindent=\parindent, labelwidth=0pt, itemindent=!]
            \item 
            \begin{align*}
            & \partial_i\partial_j\partial_\nu F_0^{X,T'}\big|_{\su_2 - f\su_1 = 0}\cdot h^{\nu\nu}\cdot \partial_\nu\partial_k\partial_l \int F_{0,1}^{X,(L, f)}
            + \partial_i\partial_j\partial_\nu \int F_{0,1}^{X,(L, f)} \cdot h^{\nu\nu}\cdot \partial_\nu\partial_k\partial_l  F_0^{X,T'}\big|_{\su_2 - f\su_1 = 0}\\
            & = \partial_j\partial_k\partial_\nu F_0^{X,T'}\big|_{\su_2 - f\su_1 = 0}\cdot h^{\nu\nu}\cdot \partial_\nu\partial_i\partial_l \int F_{0,1}^{X,(L, f)}
            + \partial_j\partial_k\partial_\nu \int F_{0,1}^{X,(L, f)} \cdot h^{\nu\nu}\cdot \partial_\nu\partial_i\partial_l  F_0^{X,T'}\big|_{\su_2 - f\su_1 = 0};
            \end{align*}            

            \item $$
            \partial_i\partial_j\partial_\nu \int F_{0,1}^{X,(L, f)}\cdot h^{\nu\nu}\cdot \partial_\nu\partial_k\partial_l \int F_{0,1}^{X,(L, f)}
            =\partial_j\partial_k\partial_\nu \int F_{0,1}^{X,(L, f)}\cdot h^{\nu\nu}\cdot \partial_\nu\partial_i\partial_l \int F_{0,1}^{X,(L, f)};
            $$

            \item $$
                \partial_i\partial_j\partial_\nu F_0^{X,T'}\cdot g^{\nu\nu}\cdot \partial_\nu\partial_k\partial_l F_0^{X,T'}
                =\partial_j\partial_k\partial_\nu F_0^{X,T'}\cdot g^{\nu\nu}\cdot \partial_\nu\partial_i\partial_l F_0^{X,T'}.
                $$
        \end{enumerate}

    \item For $i, j, k \in \{1, \dots, m\}$:
    \begin{enumerate}[label=(II\alph*), align=left, leftmargin=0pt, labelindent=\parindent, labelwidth=0pt, itemindent=!]
        \item $$
        \partial_i\partial_j\partial_\nu F_0^{X,T'}\big|_{\su_2 - f\su_1 = 0}\cdot h^{\nu\nu}\cdot \partial_\nu\partial_k\partial_o \int F_{0,1}^{X,(L, f)}
        =\partial_j\partial_k\partial_\nu F_0^{X,T'}\big|_{\su_2 - f\su_1 = 0}\cdot h^{\nu\nu}\cdot \partial_\nu\partial_i\partial_o \int F_{0,1}^{X,(L, f)};
        $$

        \item $$
        \partial_i\partial_j\partial_\nu \int F_{0,1}^{X,(L, f)}\cdot h^{\nu\nu}\cdot \partial_\nu\partial_k\partial_o \int F_{0,1}^{X,(L, f)}
        =\partial_j\partial_k\partial_\nu \int F_{0,1}^{X,(L, f)}\cdot h^{\nu\nu}\cdot \partial_\nu\partial_i\partial_o \int F_{0,1}^{X,(L, f)}.
        $$
    \end{enumerate}

    \item For $i, j \in \{1, \dots, m\}$:
    \begin{enumerate}[label=(III\alph*), align=left, leftmargin=0pt, labelindent=\parindent, labelwidth=0pt, itemindent=!]
        \item $$
        \partial_i\partial_j\partial_\nu F_0^{X,T'}\big|_{\su_2 - f\su_1 = 0}\cdot h^{\nu\nu}\cdot \partial_\nu\partial_o\partial_o \int F_{0,1}^{X,(L, f)} -\su_1 \partial_i\partial_j\partial_o \int F_{0,1}^{X,(L, f)} =0;
        $$

        \item $$
        \partial_i\partial_j\partial_\nu \int F_{0,1}^{X,(L, f)}\cdot h^{\nu\nu}\cdot \partial_\nu\partial_o\partial_o \int F_{0,1}^{X,(L, f)}
        =\partial_j\partial_o\partial_\nu \int F_{0,1}^{X,(L, f)}\cdot h^{\nu\nu}\cdot \partial_\nu\partial_i\partial_o \int F_{0,1}^{X,(L, f)}.
        $$

    \end{enumerate}
\end{enumerate}
Here, the summation index $\nu$ runs through $1, \dots, m$. Identity (Ic) is valued in $\Lambda_X^{T'}\formal{t^1, \dots, t^m}$, while all the other identities are valued in $\Lambda_X^{T_f}\formal{t^1, \dots, t^m, e^{t^o}\sX_0}$.
\end{proposition}

In particular, (Ic) recovers the WDVV equation \eqref{eqn:X-WDVV} for $X$.

\begin{proof}
The proposition directly follows from applying the expansion \eqref{eqn:F0Expansion} in Theorem \ref{thm:OpenClosed} to both sides of \eqref{eqn:tX-WDVV} and reading off appropriate coefficients, under the following rules:
\begin{itemize}
    \item For (I), apply with $i,j,k,l$ as they are.
    \item For (II), apply with $i,j,k$ as they are and $l = m+1$.
    \item For (III), apply with $i,j$ as they are and $k = l = m+1$.
    \item For (Ia), (IIa), (IIIa), read off the coefficients of $\sv^{-1}$ on both sides.
    \item For (Ib), (IIb), (IIIb), read off the coefficients of $\su_4\sv^{-2}$ on both sides.
    \item For (Ic), read off the coefficients of $\su_4^{-1}$ on both sides.
\end{itemize}
Here, we use the following observation: We have
\begin{equation}\label{eqn:tgExtra}
    \tg^{(m+1)(m+1)} = \Delta^{m+1,\tT'} = (\su_2 - f\su_1)(\su_4 + \su_2 - f\su_1) \su_1(\su_1+\su_4) = (\sv^2 \pm \su_4\sv) \su_1(\su_1+\su_4)
\end{equation}
where the sign `$\pm$' is `$+$' when $f \ge 0$ and `$-$' when $f<0$ (see \eqref{eqn:sv} for the notation $\sv$). It has second-order zeroes along $\sv, \su_4$, and thus, the $\nu = m+1$ terms in \eqref{eqn:tX-WDVV} do not contribute to the result except for case (IIIa), where the triple derivative $\partial_{m+1}\partial_{m+1}\partial_{m+1}\frac{(t^{m+1})^3}{6\Delta^{m+1,\tT'}}$ provides a cancelling factor $\frac{1}{\Delta^{m+1,\tT'}}$. Moreover, we change from $\partial_{m+1}$ to $\partial_o$ using the relation $\partial_{m+1} = -\frac{\partial_o}{\su_1}$.
\end{proof}

\begin{remark}\label{rmk:WDVVDiff} \rm{
Note that identities (IIa), (IIIa) of Proposition \ref{prop:openWDVV} resemble but are different from the open WDVV equation of \cite{HS, ST19} (stated in Theorem \ref{thm:openWDVV point}) for the disk potential with point-like boundary insertions, and the difference arises from how the $\nu = m+1$ terms contribute, as indicated in the proof above.
}
\end{remark}

\section{Frobenius structures on open Gromov-Witten theory}\label{sect:openFrob}
In this section, we use the equations in Proposition \ref{prop:openWDVV} to construct Frobenius structures on the open and closed Gromov-Witten theory of $(X,L)$, specifically:
\begin{itemize}
\item (Section \ref{sect:FrobStr}, Theorems \ref{thm:Frob}, \ref{thm:decomp}) a semi-simple formal Frobenius manifold structure on $\Spec(\Lambda_{X,L}^{T_f}[\epsilon]\formal{t^1, \dots, t^m, t^o})$ where $\epsilon$ is a nilpotent variable with $\epsilon^2 = 0$;
\item (Section \ref{sect:F-Str}, Theorem \ref{thm:F}) a flat formal $F$-manifold structure without unit on $\Spec(\Lambda_{X,L}^{T_f}\formal{t^1, \dots, t^m, t^o})$ in which the $t^o$-direction is nilpotent.
\end{itemize}
Both structures can be viewed as extensions of the semi-simple formal Frobenius manifold
\begin{equation}\label{eqn:HXRestrict}
    \hH_X^f := \Spec(\Lambda_{X,L}^{T_f}\formal{t^1, \dots, t^m}),
\end{equation}
which is obtained from $\hH_X$ by base change to $\Lambda_{X,L}^{T_f}$.

\subsection{A formal Frobenius structure}\label{sect:FrobStr}
In this section, we construct a Frobenius structure on the formal scheme
$$
    \hH_1 :=\mathrm{Spec}(\Lambda_{X,L}^{T_f}[\epsilon]\formal{t^1,\dots,t^m ,t^o})
$$
over the base ring
$$
    \Lambda_{X,L}^{T_f}[\epsilon] := \Lambda_{X,L}^{T_f} \otimes \bC[\epsilon]/\inner{\epsilon^2}.
$$
Let $\cO_{\hH_1}$ be the structure sheaf on $\hH_1$ and $\cT_{\hH_1}$ be the tangent sheaf on
$\hH_1$.
Then $\cT_{\hH_1}$ is a sheaf of free $\cO_{\hH_1}$-modules of rank $m+1$.
Given an open set $U$ in $\hH_1$, we have
$$
\cT_{\hH_1}(U)  \cong \bigoplus_{i=1}^m\cO_{\hH_1}(U) \frac{\partial}{\partial t^{i}}\bigoplus\cO_{\hH_1}(U)\frac{\partial}{\partial t^o}.
$$
We will construct a potential function $F$ involving both the open and closed Gromov-Witten invariants of $(X, L)$, as well as a pairing $(-,-)$ on $\cT_{\hH_1}$. We prove the associativity of the induced product $\star_t$ on $\cT_{\hH_1}$, which packages identities (Ia), (Ic), (IIa), and (IIIa) of Proposition \ref{prop:openWDVV}. We show that the resulting tuple $(\hH_1,\star_t,(-,-))$ is a semi-simple formal Frobenius manifold.

\subsubsection{Potential}\label{sec:potential}
Introduce the variable $\epsilon$ with $\epsilon^2=0$.
\begin{definition}\label{def:potential}
We define the \emph{potential function} $F$ as
\begin{equation}
F(t^1, \dots, t^m, t^o):= -\frac{\su_1}{6}(t^o)^3 + F_0^{X,T'}(t^1, \dots, t^m)\big|_{\su_2-f\su_1=0}+\epsilon\int F_{0,1}^{X,(L,f)}(t^1, \dots, t^m,t^o).
\end{equation}
\end{definition}

\subsubsection{Pairing}\label{sec:pairing}
In \eqref{eqn:PairingRestrict}, we defined the restriction $(h_{ij})$ of the $T'$-equivariant Poincar\'e pairing $(g_{ij})$ to $T_f$. We now extend this pairing to the $t^o$-direction. Recall that we have the change of variables $t^{m+1}=-\su_1 t^o$ from Theorem \ref{thm:OpenClosed}, which identifies $\frac{\partial}{\partial t^o}$ with $-\su_1\frac{\partial}{\partial t^{m+1}}$. Moreover, we have
$$
    \left(\frac{\partial}{\partial t^{m+1}}, \frac{\partial}{\partial t^{m+1}}\right)_{\tX,\tT'} = \tg_{(m+1)(m+1)} =\frac{1}{(\sv^2 \pm \su_4\sv) \su_1(\su_1+\su_4)}
$$
(see \eqref{eqn:tgExtra}). Clearing the second-order poles along $\sv, \su_4$, we set
$$
    h_{oo} := 1, \qquad \qquad h_{io} = h_{oi} := 0, \qquad i = 1, \dots, m.
$$

\begin{definition}\label{def:pairing}
We define the pairing $(-,-)$ on $\cT_{\hH_1}$ by the following: For any $i,j\in \{1,\dots,m,o\}$,
$$
    \left(\frac{\partial}{\partial t^i},\frac{\partial}{\partial t^j}\right) := h_{ij}.
$$
\end{definition}
As before, let $(h^{ij})=(h_{ij})^{-1}$.

\subsubsection{WDVV equations}
\begin{proposition}\label{prop:openWDVV1}
For any $i,j,k,l\in\{1,\dots,m,o\}$, the following WDVV equation holds:
\begin{equation}\label{eqn:openWDVV1}
\partial_i\partial_j\partial_\nu F\cdot h^{\nu\mu}\cdot \partial_\mu\partial_k\partial_l F
=\partial_j\partial_k\partial_\nu F\cdot h^{\nu\mu}\cdot \partial_\mu\partial_i\partial_l F
\end{equation}
where the summation indices $\nu, \mu$ run through $1, \dots, m, o$.
\end{proposition}

\begin{proof}
Note that $(h^{ij})$ is diagonal and the summation is over $\nu = \mu$. The proposition directly follows from identities (Ia), (Ic), (IIa), and (IIIa) of Proposition \ref{prop:openWDVV}, under the following rules:
\begin{itemize}
    \item When $i,j,k,l \in \{1, \dots, m\}$, the $\epsilon^0$-term of \eqref{eqn:openWDVV1} follows from identity (Ic) and the $\epsilon^1$-term follows from (Ia).

    \item When $i,j,k \in \{1, \dots, m\}$, $l = o$, there is no $\epsilon^0$-term in \eqref{eqn:openWDVV1} and the $\epsilon^1$-term follows from (IIa).

    \item When $i,j \in \{1, \dots, m\}$, $k = l = o$, again there is no $\epsilon^0$-term in \eqref{eqn:openWDVV1} and the $\epsilon^1$-term follows from (IIIa).
\end{itemize}
Any other case is either trivial or symmetric to a case above. Here, we note that since $\epsilon^2 = 0$, the equation \eqref{eqn:openWDVV1} does not contain terms involving a product of two copies of $F_{0,1}^{X,(L,f)}$ (or their antiderivatives).
\end{proof}

\subsubsection{The formal Frobenius manifold}
\begin{definition}\label{def:prod}
For any $i,j\in \{1,\dots,m, o\}$, define the product $\frac{\partial}{\partial t^i}\star_t\frac{\partial}{\partial t^j}$ on $\cT_{\hH_1}$ by
$$
\left(\frac{\partial}{\partial t^i}\star_t\frac{\partial}{\partial t^j},\frac{\partial}{\partial t^k}\right)=\frac{\partial^3F}{\partial t^i\partial t^j\partial t^k}
$$
where $k$ ranges through $1,\dots,m, o$.
\end{definition}
By Proposition \ref{prop:openWDVV1}, the product $\star_t$ is indeed associative. Moreover, it is clear by definition that we have the compatibility condition
$$
\left(\frac{\partial}{\partial t^i}\star_t\frac{\partial}{\partial t^j},\frac{\partial}{\partial t^k}\right)=\left(\frac{\partial}{\partial t^i},\frac{\partial}{\partial t^j}\star_t\frac{\partial}{\partial t^k}\right).
$$
In other words, we have the following result.

\begin{theorem}\label{thm:Frob}
The tuple $(\hH_1,\star_t,(-,-))$ is a formal Frobenius manifold over $\Lambda_{X,L}^{T_f}[\epsilon]$.
\end{theorem}

\subsubsection{Semi-simplicity of $\hH_1$}
Let $S=\cO_{\hH_1}(\hH_1)$. Consider the global Frobenius algebra $A=(\cT_{\hH_1}(\hH_1),\star_t,(,))$ and let $I\subset S$ be the ideal generated by $Q$ and $\sX_0$. Then $A$ is a free $S$-module of rank $m+1$. Let
$$
S_n:=S/I^n,\quad A_n:=A\otimes_{S}S_n.
$$
Then $A_n$ is a free $S_n$-module of rank $m+1$, and the ring structure $\star_t$ on
$A$ induces a ring structure $*_{\underline{n}}$ on $A_n$. Note that $A_1$ encodes the classical product. From the construction, the semi-simplicity of the (classical) Frobenius algebra \eqref{eqn:tXClassical} associated to $\tX$ implies that $A_1$ is semi-simple and $\{\xi_1^{(1)}:=\frac{\partial}{\partial t^1},\dots,\xi_{m}^{(1)}:=\frac{\partial}{\partial t^{m}},\xi_o^{(1)}:=\frac{\partial}{\partial t^o}\}$ is a system of idempotent basis of $A_1$. For $n\geq 1$, let $\{ \xi_{i}^{(n+1)}:i=1,\dots,m, o\}$
be the unique idempotent basis of $(A_{n+1},\star_{\underline{n+1}})$
which is the lift of the idempotent basis $\{ \xi_{i}^{(n)}:i=1,\dots,m, o\}$
of $(A_n,\star_{\underline{n}})$ \cite[Lemma 16]{LP}. Then
$$
\{ \xi_{i}(t):=\lim \xi_{i}^{(n)}: i=1,\dots,m, o\}
$$
is an idempotent basis of $(A,\star_t)$. Therefore, we have the following result.

\begin{theorem}\label{thm:decomp}
The formal Frobenius manifold $(\hH_1,\star_t,(-,-))$ is semi-simple.
\end{theorem}

\begin{remark}\label{rmk:FrobDecomp}\rm{
As discussed in Remark \ref{rmk:Frob}, the structural morphism
$$
    \hH_1 \to \Spec(\Lambda_{X,L}^{T_f}[\epsilon])
$$
may be viewed as a submersion of (formal) supermanifolds over $\Lambda_{X,L}^{T_f}$ with $\epsilon$ viewed as an odd variable. Taking $\epsilon = 0$, we obtain a Frobenius structure on the underlying reduced formal manifold, which we denote by $\hH_{1, \red}$. The induced global Frobenius algebra of $\hH_{1, \red}$ decomposes as the direct sum of the global Frobenius algebra of $\hH_X^f$ (defined in \eqref{eqn:HXRestrict}) and a 1-dimensional Frobenius algebra over $\Lambda_{X,L}^{T_f}$ generated by $\frac{\partial}{\partial t^o}$, and the decomposition is consistent with the semi-simplicity description above. In particular, $\hH_{1, \red}$ is semi-simple over $\Lambda_{X,L}^{T_f}$, and $\hH_1$ may be viewed as an infinitesimal deformation of $\hH_{1, \red}$.
}\end{remark}



\subsection{A flat formal $F$-manifold structure}\label{sect:F-Str}
In this section, we construct a flat $F$-manifold structure on the formal scheme
$$
    \hH_2 :=\mathrm{Spec}(\Lambda_{X,L}^{T_f}\formal{t^1,\dots,t^m ,t^o})
$$
over the base ring $\Lambda_{X,L}^{T_f}$, where as compared to $\hH_1$ introduced in Section \ref{sect:FrobStr}, we drop the variable $\epsilon$. Let $\cO_{\hH_2}$ be the structure sheaf on $\hH_2$ and $\cT_{\hH_2}$ be the tangent sheaf on
$\hH_2$.
Then $\cT_{\hH_2}$ is a sheaf of free $\cO_{\hH_2}$-modules of rank $m+1$.
Given an open set $U$ in $\hH_2$, we have
$$
\cT_{\hH_2}(U)  \cong \bigoplus_{i=1}^m\cO_{\hH_2}(U) \frac{\partial}{\partial t^{i}}\bigoplus\cO_{\hH_2}(U)\frac{\partial}{\partial t^o}.
$$
We will construct a vector potential $\overline{F} = (F^1, \dots, F^m, F^o)$ whose second derivatives give structural coefficients for a product $\star_t$ on $\cT_{\hH_1}$. We prove the associativity of $\star_t$, which packages identities (Ia), (Ib), (Ic), (IIa), and (IIb) of Proposition \ref{prop:openWDVV}.

\subsubsection{Vector potential}\label{sec:vector potential}
Let $(h^{ij})$ be as defined in \eqref{eqn:PairingRestrict}.
\begin{definition}
We define the \emph{vector potential} $\overline{F} = (F^1, \dots, F^m, F^o)$ by
$$
    F^i(t^1, \dots, t^m, t^o) := h^{ii} \partial_i \left(F_0^{X,T'}(t^1, \dots, t^m)\big|_{\su_2-f\su_1=0}+\int F_{0,1}^{X,(L,f)}(t^1, \dots, t^m,t^o=0)\right) 
$$
for $i = 1, \dots, m$ and
$$
    F^o(t^1, \dots, t^m, t^o) := F_{0,1}^{X,(L,f)}(t^1, \dots, t^m,t^o = 0).
$$
\end{definition}

All components of $\overline{F}$ are functions that are independent of the variable $t^o$. As discussed in Remark \ref{rmk:F}, as we set $t^o = 0$ in the definitions, conceptually $\overline{F}$ has no insertions from the open sector. The $t^o$-direction may also be viewed as an auxiliary direction in addition to the original $m$ directions; see Remark \ref{rmk:FExtension}.

\subsubsection{Open WDVV equations}
\begin{proposition}\label{prop:openWDVV2}
For any $i,j,k,l\in\{1,\dots,m, o\}$, the following open WDVV equation holds:
\begin{equation}\label{eqn:openWDVV2}
    \partial_i\partial_\mu F^j \cdot \partial_k\partial_l F^\mu
    =\partial_k\partial_\mu F^j \cdot  \partial_i\partial_l F^\mu.
\end{equation}
where the summation index $\mu$ runs through $1, \dots, m, o$.
\end{proposition}

\begin{proof}
Recall that the vector potential $\overline{F}$ consists of functions that are independent of $t^o$. Thus the two sides of \eqref{eqn:openWDVV2} are zero if at least one of $i, k, l$ is $o$. For the remaining case $i, k, l \in \{1, \dots, m\}$, first note that the term in \eqref{eqn:openWDVV2} corresponding to $\mu = o$ is again zero. Then the case $j \in \{1, \dots, m\}$ follows from identities (Ia), (Ib), and (Ic) of Proposition \ref{prop:openWDVV}, and the case $j = o$ follows from identities (IIa) and (IIb).
\end{proof}

\subsubsection{The flat formal $F$-manifold}
Let $\nabla$ be the flat connection on $\cT_{\hH_2}$ under which $\frac{\partial}{\partial t^1},\dots,\frac{\partial}{\partial t^m}$, $\frac{\partial}{\partial t^o}$ are flat. Moreover, we define the following product.

\begin{definition}
For any $i,j\in \{1,\dots,m, o\}$, define the product $\frac{\partial}{\partial t^i}\star_t\frac{\partial}{\partial t^j}$ on $\cT_{\hH_2}$ by
$$
    \frac{\partial}{\partial t^i}\star_t\frac{\partial}{\partial t^j} = \frac{\partial^2F^k}{\partial t^i\partial t^j} \frac{\partial}{\partial t^k}
$$
where the summation index $k$ runs through $1, \dots, m, o$.
\end{definition}

Since the components of the vector potential $\overline{F}$ are independent of $t^o$, the above definition implies that
$$
    \frac{\partial}{\partial t^i} \star_t \frac{\partial}{\partial t^o} = 0
$$
for any $i = 1, \dots, m, o$. Thus, $\frac{\partial}{\partial t^o}$ is nilpotent. Moreover, the product $\star_t$ does not admit an identity field, which means that the induced structure on $\hH_2$ will be an formal $F$-manifold \emph{without unit}. This is different from the case studied by \cite{HS, ST19} (see Theorem \ref{thm:openWDVV point}) and the difference is reflected by that our $F^o = F_{0,1}^{X,(L, f)}$ is supported on the ideal of $\Lambda_{X,L}$ generated by $\sX_0$ while the disk potential of \cite{HS, ST19} has a constant term. The difference is discussed from the perspective of open WDVV equations in Remark \ref{rmk:WDVVDiff}.


Summarizing the above, we arrive at the following result.

\begin{theorem}\label{thm:F}
The tuple $(\hH_2, \nabla, \star_t)$ is a flat formal $F$-manifold without the unit over $\Lambda_{X,L}^{T_f}$ in which the $t^o$-direction is nilpotent.
\end{theorem}

\begin{remark}\label{rmk:FExtension}\rm{
The flat formal $F$-manifold $\hH_2$ is a rank-1 extension of the formal Frobenius manifold $\hH_X^f$ (defined in \eqref{eqn:HXRestrict}) in the sense of e.g. \cite[Chapter 3]{Alcolado17}, \cite[Section 4]{BB19}. In other words, there is a surjective homomorphism from the global algebra of $\hH_2$ to that of $\hH_X^f$ whose kernel is the rank-1 algebra over $\Lambda_{X,L}^{T_f}$ generated by the nilpotent element $\frac{\partial}{\partial t^o}$.
}
\end{remark}

\appendix
\section{Deferred proofs}

\subsection{Proof of Theorem \ref{thm:OpenClosed}}\label{appdx:OpenClosed}
We consider the contributions of individual effective classes to the Gromov-Witten potential $F_0^{\tX,\tT'}$ of $\tX$. Let $\tbeta = (\beta, d) \in E(\tX)$, which by Section \ref{sect:CurveClass} corresponds to an effective class in $E(X,L)$. By \eqref{eqn:4FoldF0}, we consider the computation of the closed invariant
$$
    \inner{\tilde{t}, \dots, \tilde{t}}^{\tX, \tT'}_{0,n,\tbeta}
$$
by localization as detailed in \cite[Section 3.5]{LY21}, \cite[Section 3.5]{LY22} and adopt the notations there. Components of the $\tT'$-fixed locus of the moduli space $\Mbar_{0, n}(\tX, \tbeta)$ are indexed by the set $\Gamma_{0, n}(\tX, \tbeta)$ of decorated graphs (see \cite[Section 3.1]{LY21}, \cite[Section 3.2]{LY22}). We have
\begin{equation}\label{eqn:closedGraphSum}
    \inner{\tilde{t}, \dots, \tilde{t}}^{\tX,\tT'}_{0,n,\tbeta} = \sum_{\tGa \in \Gamma_{0,n}(\tX, \tbeta)} \tC_{\tGa}
\end{equation}
where $\tC_{\tGa}$ is the contribution of the component indexed by $\tGa$. 

Let $\tGa = (\Gamma, \vf, \vd, \vs) \in \Gamma_{0,n}(\tX, \tbeta)$. As in \cite[Section 4.3]{LY22}, let
$$
    V_0 := \{v \in V(\Gamma) : \vf(v) \in \iota(\Sigma(3)) \}, \qquad     E_2 := \{e \in E(\Gamma) : \vf(e) = \iota(\tau_0)\},
$$
and $c_0$ denote the number of connected components of the subgraph of $\Gamma$ induced on $V_0$. We may assume that $c_0 \ge 1$, since otherwise, $\tGa$ represents a constant map to the fixed point $\tp_{m+1}$\footnote{The contribution of such maps to $F_0^{\tX,\tT'}$ has already been singled out in the term $\frac{(t^{m+1})^3}{6\Delta^{m+1,\tT'}}$ in \eqref{eqn:F0Expansion}.} and thus $\tC_{\tGa} = 0$. By the proof of \cite[Lemma 4.4]{LY22}, the total power of $\su_4$ in $\tC_{\tGa}$ is
$$
    |E_2| - c_0 \ge -1.
$$
Equality holds if and only if $E_2 = \emptyset$ and $V_0 = V(\Gamma)$, which happens if and only if $d = 0$. Therefore, $F_0^{\tX,\tT'}$ has at most a simple pole along $\su_4$ and the residue $\tA$ is supported on the Novikov variables $\{\tQ^{\iota_*(\beta)}: \beta \in E(X)\}$ and is independent of $t^{m+1}$. Part (b) of the theorem follows from the following result.

\begin{lemma}\label{lem:closed}
For $\tbeta = (\beta, 0)$, we have
$$
    \su_4\inner{\tilde{t}, \dots, \tilde{t}}^{\tX,\tT'}_{0,n,\tbeta} \big|_{\su_4=0} = \inner{t, \dots, t}^{X,T'}_{0,n,\beta}.
$$
\end{lemma}

\begin{proof}
We consider the contributions from decorated graphs as in \eqref{eqn:closedGraphSum}. In the case $\tbeta = \iota_*(\beta)$, any $\tGa \in \Gamma_{0, n}(\tX, \tbeta)$ (with $c_0 \neq 0$) naturally corresponds to a decorated graph in $\Gamma_{0, n}(X, \beta)$; that is, it represents stable maps which factor through $X \subset \tX$. The lemma then follows directly from the comparison of localization contributions as in the proof of \cite[Lemma 4.2]{LY22}. Note from \eqref{eqn:FixedPtRestrict} that $\tilde{t} \big|_{\tp_i} = t \big|_{p_i}$ for $i = 1, \dots, m$.
\end{proof}

Now we consider the case $d>0$ which corresponds to the part of $F_0^{\tX,\tT'}$ that does not have a pole along $\su_4$. By the divisor equation, we have
$$
    \inner{\tilde{t}, \dots, \tilde{t}}^{\tX, \tT'}_{0,n,\tbeta} = \frac{1}{d} \inner{\tilde{t}, \dots, \tilde{t}, \tD}^{\tX, \tT'}_{0,n+1,\tbeta}.
$$
Similar to \eqref{eqn:closedGraphSum}, we consider the localization computation of this invariant as a sum of contributions from decorated graphs:
\begin{equation}\label{eqn:openGraphSum}
    \inner{\tilde{t}, \dots, \tilde{t}, \tD}^{\tX,\tT'}_{0,n+1,\tbeta} = \sum_{\tGa \in \Gamma_{0,n+1}(\tX, \tbeta)} \tC_{\tGa}
\end{equation}
where by an abuse of notation $\tC_{\tGa}$ denotes the contribution of $\tGa \in \Gamma_{0,n+1}(\tX, \tbeta)$. We study the poles of $\tC_{\tGa}$ along $\su_2 - f\su_1$ or $\su_2 - f\su_1 - \su_4$. We assume below that $f \in \bZ$ is generic with respect to the curve class $\tbeta$. Eventually, the argument in \cite[Section 4.4]{LY22} will enable us to extend the proof to all $f \in \bZ$.

Note that $E_2 \neq \emptyset$ when $d>0$. By the computations in the proof of \cite[Lemma 4.6]{LY22}, we can write
$$
    \tC_{\tGa} = \begin{cases}
            \frac{\tbw(\ttau_3, \tsi_0)^{|E_2|-1}}{\tbw(\ttau_2, \tsi_0)} \tc_{\tGa} & \text{if $f \ge 0$}\\
            \frac{\tbw(\ttau_2, \tsi_0)^{|E_2|-1}}{\tbw(\ttau_3, \tsi_0)}\tc_{\tGa} & \text{if $f < 0$}
        \end{cases}
        = \frac{1}{\sv}\tb_{\tGa} + \frac{\su_4}{\sv}\tc_{\tGa,1} + \tc_{\tGa,2}
$$
where each of $\tc_{\tGa}, \tb_{\tGa}, \tc_{\tGa,1}, \tc_{\tGa,2}$ has a well-defined weight restriction to $\su_4 = 0, \su_2 - f\su_1 = 0$. Moreover, $\tb_{\tGa}$ is nonzero only if $|E_2|=1$, in which case, \cite[Lemma 4.2]{LY22} implies that $\tb_{\tGa}$ (or the graph $\tGa$) contributes to the localization computation of a corresponding disk invariant of $(X,L,f)$. More formally, and combining the analysis over all decorated graphs, we have the following lemma which is a direct consequence of \cite[Lemma 4.2]{LY22}.

\begin{lemma}\label{lem:open}
For $\tbeta = (\beta, d)$ with $d > 0$, we can write
$$
    \inner{\tilde{t}, \dots, \tilde{t}, \tD}^{\tX,\tT'}_{0,n+1,\tbeta} = \frac{1}{\sv}\tb + \frac{\su_4}{\sv}\tc_1 + \tc_2
$$
where each of $\tb, \tc_1, \tc_2$ has a well-defined weight restriction to $\su_4 = 0, \su_2 - f\su_1 = 0$ and
$$
    \tb \big|_{\su_4 = 0, \su_2-f\su_1 =0} = \inner{t, \dots, t}^{X, (L,f)}_{(0,1),n,\beta+d[B], d}.
$$
\end{lemma}

Lemma \ref{lem:open} implies part (c) of Theorem \ref{thm:OpenClosed} and completes the proof.


\begin{thebibliography}{AA}





\bibitem{Alcolado17} A. Alcolado, 
{\em Extended Frobenius manifolds and the open WDVV equations},
Ph.D. thesis, McGill University, ProQuest LLC, 2017.

\bibitem{AL23} K. Aleshkin, C.-C. M. Liu, 
``Open/closed correspondence and extended LG/CY correspondence for quintic threefolds,''
{\tt arXiv:2309.14628}.

\bibitem{ABLR20} A. Arsie,  A. Buryak, P. Lorenzoni, P. Rossi, 
``Semisimple flat $F$-manifolds in higher genus,'' 
Commun. Math. Phys. {\bf 397} (2023), 141--197.



\bibitem{BB19} A. Basalaev, A. Buryak, 
``Open WDVV equations and Virasoro constraints,'' 
Arnold Math. J. {\bf 5} (2019), no. 2--3, pp. 145--186.





\bibitem{BBvG24} P. Bousseau, A. Brini, M. van Garrel,
``Stable maps to Looijenga pairs,''
Geom. Topol. {\bf 28} (2024), no. 1, 393--496.






\bibitem{BT24} R. Blumberg, S. Tukachinsky,
``WDVV based recursion for open Gromov-Witten invariants,'' {\tt arXiv: 2402.10542}.


\bibitem{BCT18} A. Buryak, E. Clader, R. Tessler, ``Open $r$-spin theory II: The analogue of Witten's conjecture for $r$-spin disks,''
J. Differential Geom. {\bf 128} (2024), no. 1, 1--75.

\bibitem{BCT19} A. Buryak, E. Clader, R. Tessler, ``Closed extended $r$-spin theory and the Gelfand-Dickey wave function,'' 
J. Geom. Phys. {\bf 137} (2019), 132--153.

\bibitem{BG23} A. Buryak, D. Gubarevich, 
``Integrable systems of finite type from $F$-cohomological field theories without unit,'' 
Math. Phys. Anal. Geom. {\bf 26} (2023), 23.


\bibitem{CZ19} X. Chen, A. Zinger, 
``WDVV-Type relations for disk Gromov--Witten invariants in dimension 6,'' 
Math. Ann. {\bf 379} (2021), no. 3-4, pp. 1231--1313.








\bibitem{Dub96} B. Dubrovin, 
``Geometry of 2D topological field theories,'' 
in {\it Integrable systems and quantum groups
(Montecatini Terme, 1993)}, 120--348, Lecture Notes in Math., 1620, Fond. CIME/CIME Found.
Subser., Springer, Berlin, 1996.

\bibitem{Dub99} B. Dubrovin, 
``Painlev\'e transcendents in two-dimensional topological field theory,'' 
in {\it The Painlev\'e property}, 287--412, CRM Ser. Math. Phys., Springer, New York, 1999.


\bibitem{FL13} B. Fang, C.-C. M. Liu,
``Open Gromov-Witten invariants of toric Calabi-Yau 3-folds,''
Commun. Math. Phys. {\bf 323} (2013), 285--328.

\bibitem{FLT22} B. Fang, C.-C. M. Liu, H.-H. Tseng,
``Open-closed Gromov-Witten invariants of 3-dimensional Calabi-Yau smooth toric DM stacks,''
Forum Math. Sigma {\bf 10} (2022), Paper No. e58, 56 pp.





\bibitem{FOOO09} K. Fukaya, Y.-G. Oh, H. Ohta, K. Ono, 
{\it Lagrangian intersection Floer theory: anomaly and obstruction. Part I}, 
AMS/IP Studies in Advanced Mathematics, vol. 46, American Mathematical Society, Providence, RI; International Press, Somerville, MA, 2009.




\bibitem{vGGR19} M. van Garrel, T. Graber, H. Ruddat, ``Local Gromov-Witten invariants are log invariants,'' Adv. Math. \textbf{350} (2019), 860--876.




\bibitem{Ganatra23} S. Ganatra,
``Cyclic homology, $S^1$-equivariant Floer cohomology and Calabi-Yau structures,'' Geom. Topol. {\bf 27} (2023), no. 9, 3461--3584.


\bibitem{GPS15} S. Ganatra, T. Perutz, and N. Sheridan, ``Mirror symmetry: from categories to curve counts,'' {\tt arXiv:1510.03839}.


\bibitem{GZ17} P. Georgieva and A. Zinger, ``Enumeration of real curves in $\bC P^{2n-1}$ and a
Witten-Dijkgraaf-Verlinde-Verlinde relation for real Gromov-Witten invariants,'' Duke Mathematical
Journal \textbf{166} (2017), no. 17, 3291--3347.

\bibitem{Getzler04} E. Getzler,
``The jet-space of a Frobenius manifold and higher-genus Gromov-Witten invariants,''
in {\it Frobenius manifolds}, 45--89, Aspects Math., E36, Friedr. Vieweg, Wiesbaden, 2004.

\bibitem{GP98} L. G\"{o}ttsche, R. Pandharipande, 
``The quantum cohomology of blowups of $P^2$ and enumerative geometry,'' 
J. Differential Geom. {\bf 48} (1998), no. 1, 61--90.


\bibitem{GP99} T. Graber, R. Pandharipande, ``Localization of virtual classes,'' Invent. Math. {\bf 135} (1999), no. 2, 487--518.


\bibitem{GRZ22} T. Gr\"{a}fnitz, H. Ruddat, E. Zaslow,
``The proper Landau-Ginzburg potential is the open mirror map,''
Adv. Math. \textbf{447} (2024), Paper No. 109639, 69 pp.

\bibitem{HM99} C. Hertling, Yu. Manin, 
``Weak Frobenius manifolds,'' 
Int. Math. Res. Not. IMRN {\bf 1999}
(1999), no. 6, 277--286.

\bibitem{HKSSS23} A. Hollands, E. Kosloff, M. Sela, Q. Shu, and J. P. Solomon, ``Relative quantum cohomology of the Chiang Lagrangian,'' {\tt arXiv:2305.03016}.


\bibitem{HS} A. Horev, J. Solomon, ``The open Gromov-Witten-Welschinger theory of blowups of the projective plane,'' 
{\tt arXiv:1210.4034}.

\bibitem{H22} K. Hugtenburg, ``Open Gromov-Witten invariants from the Fukaya category,'' 
Adv. Math. \textbf{441} (2024), Paper No. 109542, 44 pp.

\bibitem{HT24} K. Hugtenburg, S. Tukachinsky, ``Examples of relative quantum cohomology,'' 
{\tt arXiv:2402.03209}.








  

\bibitem{KM94} M. Kontsevich, Yu. Manin, ``Gromov-Witten classes, quantum cohomology, and enumerative geometry,'' Commun. Math. Phys. {\bf 164} (1994), 525--562.

\bibitem{LP} Y.-P. Lee, R. Pandharipande, 
{\em Frobenius manifolds, Gromov-Witten theory, and Virasoro constraints}, 2004. 




\bibitem{LLLZ09} J. Li, C.-C. M. Liu, K. Liu, J. Zhou, ``A mathematical theory of the topological vertex,'' Geom. Topol. {\bf 13} (2009), no. 1, 527--621.




\bibitem{Liu13} C.-C. M. Liu, ``Localization in Gromov-Witten theory and orbifold Gromov-Witten theory,''
in {\em Handbook of moduli. Vol. II}, 353--425, Adv. Lect. Math. (ALM) {\bf 25}, International Press, Somerville, MA, 2013.

\bibitem{LY21} C.-C. M. Liu, S. Yu, 
``Open/closed correspondence via relative/local correspondence,''
Adv. Math. \textbf{410} (2022), Paper No. 108696, 43 pp.

\bibitem{LY22} C.-C. M. Liu, S. Yu, 
``Orbifold open/closed correspondence and mirror symmetry,''
{\tt arXiv:2210.11721}.



\bibitem{Manin99} Yu. Manin, 
{\it Frobenius manifolds, quantum cohomology, and moduli spaces}, 
American Mathematical
Society Colloquium Publications, vol. 47, American Mathematical Society, Providence, RI, 1999.

\bibitem{Manin05} Yu. Manin,  ``$F$-manifolds with flat structure and Dubrovin's duality,''  Adv. Math. \textbf{198} (2005), 5--26.

\bibitem{MM97} Yu. Manin, S. Merkulov,
``Semisimple Frobenius (super)manifolds and quantum cohomology of $\bP^r$,'' Topol. Methods Nonlinear Anal. \textbf{9} (1997), no. 1, 107--161.


\bibitem{Mayr01} P. Mayr, ``$N = 1$ mirror symmetry and open/closed string duality,'' \texttt{arXiv:hep-th/0108229}.

\bibitem{MS94} D. McDuff, D. Salamon, 
{\it $J$-holomorphic curves and quantum cohomology}, University Lecture Series, 6, American Mathematical Society, Providence, RI, 1994.




\bibitem{PST14} R. Pandharipande, J. Solomon, R. Tessler, ``Intersection theory on moduli of disks, open KdV
and Virasoro,'' 
{\tt arXiv:1409.2191}.


\bibitem{RT94} Y. Ruan, G. Tian, ``A mathematical theory of quantum cohomology,'' Math. Res. Lett. {\bf 1} (1994), no. 2, 269--278.

\bibitem{Solomon07} J. Solomon, ``A differential equation for the open Gromov-Witten potential,'' preprint, 2007.


\bibitem{ST19} J. Solomon, S. Tukachinsky, ``Relative quantum cohomology,''
J. Eur. Math. Soc. \textbf{26} (2024), no. 9, 3497--3573.



\bibitem{Yu23} S. Yu, 
``Open/closed BPS correspondence and integrality,''
Commun. Math. Phys.  {\bf 405}, 219 (2024), 34 pp.



\end{thebibliography}
\end{document}